\documentclass{article}
\usepackage{geometry}
\usepackage{cite}
\usepackage{hyperref}
\usepackage{graphicx} 
\title{Asymptotic flocking dynamics of Relativistic-Cucker-Smale particles immersed in incompressible Navier-Stokes equations}
\author{Shenglun Yan\, and \, Weiyuan Zou\footnote{E-mail: zwy@amss.ac.cn(W.-Y. Zou)}}
\usepackage{amsthm,amsmath,amsfonts,mathrsfs}
\numberwithin{equation}{section} 
\allowdisplaybreaks[4]
\begin{document}
	\theoremstyle{definition}\newtheorem{thm}{Theorem}[section]
	\theoremstyle{definition}\newtheorem{rem}{Remark}[section]
	\theoremstyle{definition}\newtheorem{lem}{Lemma}[section]
	\theoremstyle{definition}\newtheorem{df}{Definition}[section]
	\theoremstyle{definition}\newtheorem{nt}{Notation}[section]
	\theoremstyle{definition}\newtheorem{prop}{Proposition}[section]
	\theoremstyle{remark}\newtheorem*{prof1}{Proof of Theorem 1.1}
	\maketitle
	\emph{College of Mathematics and Physics,
		Beijing University of Chemical Technology, Beijing 100029, P. R. China} 
	\begin{abstract}
		In this paper, we propose a coupled system describing the interaction between the Relativistic Cucker-Smale model and the incompressible Navier-Stokes equations via a drag force, and establish a global existence theory as well as the time-asymptotic behavior of the proposed model in $\mathbb{T}^3$. It is shown that the coupled system exhibits an exponential alignment under some specific assumptions, and that weak solutions exist globally for general initial data.
	\end{abstract}
	\textbf{Key Words.} Relativistic-Cucker-Smale model, incompressible Navier-Stokes equation, asymptotic stability, weak solution.
	\section{\label{sec:level1}Introduction}
	
	Flocking in complex systems is a common phenomenon in nature. For example, animal migration\cite{ref1}, bacterial movement\cite{ref2}, and synchronization between cells\cite{ref3} and fireflies\cite{ref4} exhibit certain clustering effects. The study of cluster effects is also of great significance. For example, we can apply the theory of cluster effects to modern fields such as sensor networks, robot control, and unmanned flight\cite{ref5}. In this paper, we consider a coupled system that describes the interaction between a large number of particles and an incompressible fluid through frictional force, namely, Relativistic-Cucker-Smale-Navier-Stokes equations (in short, RCS-NS equations).  Assuming $x\in\mathbb{T}^3$ and $w\in\mathbb{R}^3$ represent the position and relativistic velocity of particles, $f (t, x, w)$ represents the distribution function of particles, $u\in\mathbb{R}^3$ represents the velocity of fluids, and $p\in\mathbb{R}$ represents pressure, then the RCS-NS equations have the following form:
	\begin{equation}\label{eq1}
		\small{	\left\{
			\begin{aligned}
				&\partial_t f+\hat{v}(w) \cdot \nabla_x f+\nabla_w \cdot (F[f,u]f)=0,\\
				&\partial_t u+(u\cdot\nabla_x)u+\nabla_x p-\mu \Delta_x u=-\int_{\mathbb{R}^3}F_d[u]f \mathrm{d}w,\\
				&\nabla_x\cdot u=0, (t,x,w)\in \mathbb{R}_+\times\mathbb{T}^3\times\mathbb{R}^3,\quad F[f,u]:=L[f]+F_d.
			\end{aligned}
			\right.}
	\end{equation}
	subject to initial data:
	$$(f(0,x,w),u(0,x))=(f^{\text{in}}(x,w),u^{\text{in}}(x)),\quad (x,w)\in \mathbb{T}^3\times\mathbb{R}^3,$$
	and the periodic boundary conditions on $\mathbb{T}^3=\mathbb{R}^3/\mathbb{Z}^3=[0,1]^3$:
	$$u^{\text{in}}|_{\Omega_j}=u^{\text{in}}|_{\Omega_{j+3}},\;\nabla_xu^{\text{in}}|_{\Omega_j}=\nabla_xu^{\text{in}}|_{\Omega_{j+3}},\;f^{\text{in}}|_{\Omega_j}=f^{\text{in}}|_{\Omega_{j+3}},$$
	with $\Omega_j:=\partial\mathbb{T}^3\cap\{x_j=0\}$, and $\Omega_{j+3}:=\partial\mathbb{T}^3\cap\{x_j=1\}(j=1,2,3)$. Here, the relativistic velocity $w$ is defined by mapping $\hat{w}:\ B_c\left( 0 \right) \rightarrow \mathbb{R}^3$:
	$$w=\hat{w}\left( v \right) =\frac{v}{\sqrt{1-\frac{|v|^2}{c^2}}}+\frac{v}{c^2-|v|^2}.$$
	The meaning of $v$ here is the classical speed of motion of an object, and it is generally believed that $|v|<c$, where $c$ is the speed of light. It can be proved that the mapping $\hat{w}$ is reversible. We denote the inverse mapping as $\hat{v}$. The definitions of other functionals in the system of equations are as follows:
	$$L\left[ f \right] \left( t,x,w \right) :=-\int_{\mathbb{T}^3\times \mathbb{R}^3}{\phi}\left( |x-x_*| \right) \left( \hat{v}\left( w \right) -\hat{v}\left( w_* \right) \right) f\left( t,x_*,w_* \right) \mathrm{d}x_*\mathrm{d}w_*,$$
	$$F_d\left[ u \right] \left( t,x,w \right) :=u\left( t,x \right) -w\text{, and }F\left[ f,u \right] :=L\left[ f \right] +F_d.$$
	Here $L[f]$ and $F_d$ represent velocity alignment forces, and drag force per unit mass, respectively. The communication weight of the particles $\phi$ is non-decreasing in $\mathbb{R}_+$ and belongs in $W^{1,\infty}(\mathbb{R}; \mathbb{R}_+).$ 
	
	There is extensive literature on flocking in mathematical theory. In 2007, Cucker and Smale established the Cucker-Smale model based on Newton's laws by abstracting birds into particle models to address the aggregation phenomenon observed during bird migration. This is an important ordinary differential equation model for studying cluster effects\cite{ref6}. However, the Cucker-Smale model also has its own limitations. The reason is that the Cucker-Smale model focuses mainly on the interaction mechanisms between individuals within a group, but ignores the significant impact of external environmental factors such as light sources \cite{ref7,ref8} on cluster behavior. In 2017, Professor Seung-Yeal Ha's team derived a thermodynamic Cucker-Smale model considering temperature effects based on the original Cucker-Smale model in \cite{ref6}, by using gas kinetics theory. Reference \cite{ref9} reveals the stability of the thermodynamic Cucker-Smale model under initial conditions, as well as its exponential stability characteristics at temperatures far above the system's kinetic energy, providing strong theoretical support for understanding the thermodynamic behavior in complex cluster systems.
	
	As is well known, there are a large number of microscopic particles and macroscopic matter moving at speeds close to the speed of light in the universe, such as stars, nebulae, and other celestial bodies. When we study the motion of matter in these universes, we generally take into account the effects of relativity, since Newton's laws can only describe the motion of matter at macroscopic low speeds well. In 2020, Professor Ha and his team derived the relativistic Euler equations considering relativistic effects based on the theory of relativistic gas dynamics, as well as the relativistic thermodynamic Cucker-Smale model, a variant of the thermodynamic Cucker-Smale model\cite{ref10}. 
	
	Assuming $x_i,\ v_i\in \mathbb{R}^d$ and $T_i \in \mathbb{R}_+$ represent the position, velocity, and temperature of the $i$-th particle, $\phi$ and $\zeta$ are non negative uni-variate function used to characterize the communication weights between different particles, $c$ represents the speed of light, $t$ represents time, and $\displaystyle{\Gamma _i:=\frac{1}{\sqrt{1-\frac{|v_i|^2}{c^2}}}}$ is the Lorentz factor, then based on the original symbol system, the relativistic thermodynamic Cucker-Smale model takes the following form:
	\begin{equation}
		\small{	\left\{
			\begin{aligned}
				&\displaystyle{\frac{\mathrm{d}x_i}{\mathrm{d}t}=v_i,\quad i\in \left[ N \right] :=\left\{ 1,2,...,N \right\} ,\quad t>0,}\\
				&\displaystyle{\frac{\mathrm{d}}{\mathrm{d}t}\left( \Gamma _iv_i\left( 1+\frac{T_i}{c^2} \right) \right) =\frac{1}{N}\sum_{j=1}^n{\phi}\left( |x_i-x_j| \right) \left( \frac{v_j\Gamma _j}{T_j}-\frac{v_i\Gamma _i}{T_i} \right),}\\
				&\displaystyle{\frac{\mathrm{d}}{\mathrm{d}t}\left( T_i\Gamma _i+c^2\left( \Gamma _i-1 \right) \right) =\frac{1}{N}\sum_{j=1}^n{\zeta}\left( |x_i-x_j| \right) \left( \frac{\Gamma _i}{T_i}-\frac{\Gamma _j}{T_j} \right).}
			\end{aligned}
			\right.}
	\end{equation}
	The author of Reference \cite{ref10} also demonstrated the asymptotic stability of the Cucker-Smale model in relativistic thermodynamics.
	
	If the relativistic temperature of a particle $\displaystyle{\frac{T_i}{\Gamma _i}=:T^*}$ is set to a constant value and taken as 1, a more trivial relativistic Cucker-Smale model can be immediately obtained:
	\begin{equation}
		\small{	\left\{
			\begin{aligned}
				&\displaystyle{\frac{\mathrm{d}x_i}{\mathrm{d}t}=v_i,\quad i\in \left[ N \right] :=\left\{ 1,2,...,N \right\} ,\quad t>0,}\\
				&\displaystyle{\frac{\mathrm{d}}{\mathrm{d}t}\left( \Gamma _iv_i\left( 1+\frac{\Gamma _i}{c^2} \right) \right) =\frac{1}{N}\sum_{j=1}^n{\phi}\left( |x_i-x_j| \right) \left( v_j-v_i \right).}
			\end{aligned}
			\right.}
	\end{equation}
	
	The Cucker-Smale model of relativity has attracted deeper research due to its ability to more rigorously characterize the clustering effects of microscopic particles or high-speed moving particles in physics. Reference \cite{ref11} first provides a more general expression of the relativistic Cucker-Smale model. Firstly, the author introduces the relativistic velocity $w$, where the relativistic velocity $w_i$ of the $i$-th particle is defined by the mapping $\hat{w}:\ B_c\left( 0 \right) \rightarrow \mathbb{R}^d$:
	$$
	w_i=\hat{w}\left( v_i \right) =\frac{v_i}{\sqrt{1-\frac{|v_i|^2}{c^2}}}+\frac{v_i}{c^2-|v_i|^2}.
	$$
	It can be proven that the mapping $\hat{w}$ is reversible and its inverse mapping is denoted as $\hat{v}$. Reference \cite{ref11} also studied the asymptotic stability and mean-field limit of the relativistic Cucker-Smale model at distance.
	
	Considering that many forces in nature (such as universal gravitation and Coulomb force) are in inverse square form, i.e. $\displaystyle{\frac{k}{r^2}}$ (here $k$ is constant and $r$ denotes the distance between two objects), assuming that $
	\phi \left( r \right) $ is in the form of a singular kernel, like $\displaystyle{\phi(r):=\frac{1}{r^{\beta}}\left( \beta >0 \right)}$, also has significant importance. However, because singular kernels usually do not possess excellent properties such as boundedness and integrability on the entire positive real axis, the study of singular kernels is much more difficult compared to ordinary alternating weights. In \cite{ref12}, the author proved that under specific initial conditions, the relativistic Cucker-Smale model under the action of a singular kernel exhibits asymptotic stability when the exponent of singular kernel is present, that is, it physically exhibits a clustering effect.
	
	In practice, there is often a delay between the reception of information signals and the transmission of control signals. In this case, we need to incorporate the delay as a parameter into our system analysis process. In \cite{ref13}, the author proposed the delayed relativistic Cucker-Smale model and analyzed its asymptotic stability. 
	
	The fluid-particle coupling model can be used to simulate the movement process of microscopic particles dispersed in macro fluid, such as sedimentation phenomenon analysis, combustion theory, and spray modeling \cite{ref14,ref15,ref16,ref17,ref18,ref19,ref20,ref21}, and has wide applications in many modern industrial fields, such as biotechnology, medical science, mineral processing, rocket propulsion, and sewage treatment \cite{ref19,ref20,ref21,ref22,ref23,ref24,ref25}. When aggregated particles are surrounded by viscous fluid, they will generate interaction forces, such as friction \cite{ref26,ref27}.
	
	Regarding Cucker-Smale type systems, we typically consider their coupling structures with isentropic incompressible fluids and isentropic compressible fluids. Reference \cite{ref28} formally derived the coupling structure between the kinetic Cucker-Smale model and the incompressible Navier-Stokes equation and proved that the smooth solution of the Cucker-Smale-Navier-Stokes system exhibits exponential stability under certain conditions on the periodic region $\mathbb{T}^d:=\mathbb{R}^d / \mathbb{Z}^d$. Drawing on the idea of \cite{ref29}, it has been proved that the existence of weak solutions can be guaranteed when the initial value satisfies certain conditions when $d=3$. Reference \cite{ref30} proved the existence of strong solutions for the Cucker-Smale-Navier-Stokes system in $\mathbb{T}^3$.
	
	In \cite{ref31}, the author proved the asymptotic stability of the thermodynamic Cucker-Smale-Navier-Stokes system in $\mathbb{T}^3$, as well as the existence of some strong and weak solutions. In \cite{ref32}, the author studied the existence, large-time behavior, and non-relativistic limits of the RCS model.
	
	It should be noted that according to our knowledge, there is currently no research related to the relativistic Cucker-Smale-Navier-Stokes model. However, given the large number of research cases on the coupling structure of the Cucker-Smale model and its variants with incompressible Navier-Stokes equations, it is highly feasible to derive and analyze the relativistic Cucker-Smale-Navier-Stokes equations. The potential application background of this model is to characterize the propagation behavior dynamics of bio-electric signals in neural cell fluids\cite{ref33}.
	
	In this paper, we will investigate the large-time behavior and the existence of weak solutions of RCS-NS system. For the large-time behavior of RCS-NS system, we construct Lyapunov functionals and estimate its dissipative structure, obtain the corresponding $Gr\ddot{o}nwall's$ inequality, and ultimately obtain the conclusion of exponential convergence. For the existence of weak solutions: we first regularize the original system of equations; then we construct a Cauchy iteration sequence to prove the existence of approximate system; Finally, by using methods such as local extension, we obtain the approximate solution can converge to a weak solution.
	
	\begin{nt}
		For functions $f=f(x,w)$ and $u=u(x)$, $\|f\|_{L^p}$ and $\|u\|_{L^p}$ denote the usual $L^p(\mathbb{T}^3\times\mathbb{R}^3)$-norm and $L^p(\mathbb{T}^3)$-norm, respectively. We also denote by $C$ a generic positive constant. For the vector $u\in\mathbb{R}^n$, we denote $u^i$ as its $i$-th component. For two matrices $A=(a_{ij})_{n\times n}$ and $B=(b_{ij})_{n\times n}$, we define the inner product of the matrix as $A:B=\sum_{i,j}a_{ij}b_{ij}$. For simplicity, we often drop the $x$-dependence of differential operators $\nabla_x$ and $\Delta_x$, that is, $\nabla f:=\nabla_xf$ and $\Delta f:=\Delta_xf.$ For any nonnegative integer $k$, $H^k$ denote the $k$-th order $L^2$ Sobolev space. $\mathcal{C}^k([0,T];E)$ is the set of $k$-times continuously differentiable functions from an interval $[0,T]\subset\mathbb{R}$ into a Banach space $E$, and $L^p(0,T;E)$ is the set of $L^p$ functions from an interval $(0,T)$ to a Banach space $E$. $\nabla^k$ denotes any partial derivative $\partial_\alpha$ with multi-index $\alpha,|\alpha|=k.$
	\end{nt}
	
	For the asymptotic stability of \eqref{eq1}, we define the average momentum of relativistic Cucker-Smale particle and the incompressible fluid, respectively,
	$$ w_c(t):=\frac{\int_{\mathbb{T}^3\times\mathbb{R}^3}w f \mathrm{d}x\mathrm{d}w}{\int_{\mathbb{T}^3\times\mathbb{R}^3}f \mathrm{d}x\mathrm{d}w},\quad u_c(t)=\int_{\mathbb{T}^3}u\mathrm{d}x,$$
	and the compact support of $f(t,x,w)$ onto relativistic velocity spaces, as following,
	$$\Omega_{w}(f,t):=\overline{\{w:\exists x \in\mathbb{T}^3\text{ such that }f(t,x,w)\neq0\}}.$$
	Since the control equation of $f(t,x,w)$ is a quasi-linear PDE, the propagation speed of its characteristic lines is limited. That is, if $\Omega_{w}(f,0)$ is bounded, $\Omega_{w}(f,t)$ is also bounded at any moment $t \in \mathbb{R}_+$.
	
	We then define the following Lyapunov function,
	$$\begin{aligned}
		\mathcal{L}(t)&:=\mathcal{E}^w(t)+\mathcal{E}^u(t)+\mathcal{E}^r(t)\nonumber\\
		&=\int_{\mathbb{T}^3 \times \mathbb{R}^3}|w-w_c(t)|^2f(t,x,w)\mathrm{d}x\mathrm{d}w+\int_{\mathbb{T}^3}|u(t,x)-u_c(t)|^2\mathrm{d}x+\frac{1}{2}|u_c-w_c|^2.\nonumber
	\end{aligned}$$
	By the way, Reference \cite{ref10} point out the local energy function of particle $E:B_c(0) \rightarrow \mathbb{R}_+$ as
	$$E(v)=c^2(\Gamma(v)-1)+(\Gamma^2(v)-\log \Gamma(v)).$$
	
	The asymptotic stability of the RCS-NS system is expressed as the following theorem.
	\begin{thm}[Asymptotic stability of RCS-NS system]\label{thm1.1}
		Let $[f (t, x, w),u(t,x)]$ be a smooth solution to \eqref{eq1} satisfying
		
		(1)\quad $\displaystyle{\int_{\mathbb{T}^3 \times \mathbb{R}^3}f^{\text{in}}(x,w)\mathrm{d}x\mathrm{d}w=1,\quad \bigcup\limits_{t\geq 0} \Omega_w (f,t)\subset B_{W_0}(0)\quad \text{for some} \quad W_0>0}$,
		
		(2)\quad $\displaystyle{||\rho_f||_{L^\infty\left(\mathbb{R}_+\times \mathbb{T}^3\right)}<\infty,\quad \text{where} \quad \rho_f(t,x):=\int_{\mathbb{R}^3}f(t,x,w)\mathrm{d}w}$, and $f(t,x,w)$ decays to zero sufficiently fast in the $w$-variable. Then, we have
		\begin{equation}
			\mathcal{L}(t)\leq \mathcal{L}(0)e^{-Ct},
		\end{equation}
		where $C:=C(\rho_f,\mu)$ is a positive constant independent of $t$.
	\end{thm}
	Theorem \ref{thm1.1} tells us that the Lyapunov function $\mathcal{L}(t)$ will decay to zero as $t \rightarrow\infty$ under some specific initial data. This implies that the velocity of particles and fluid will tend to be consistent. In fact, we have
	$$u_c(t),w_c(t)\to\frac12\left(w_c(0)+u_c(0)\right)\quad\mathrm{~as~}\quad t\to\infty.$$
	from the conservation of the total momentum.
	\begin{rem}
		The argument used for the proof of Theorem \ref{thm1.1} can be easily extended to the $d$-dimensional case with $d\geq3$. For other Cucker-Smale type equations coupled with incompressible Navier-Stokes equations, similar estimates can also be seen in \cite{ref28,ref30,ref31}.
	\end{rem}
	\begin{rem}
		From \cite{ref31}, we know that if $\rho_f \in L^\infty(\mathbb{R}_+;L^{3/2}(\mathbb{T}^3))$, similar estimates can be reached by using $H\ddot{o}der$ inequality and the Gagliardo-Nirenberg-Sobolev inequality. Moreover, if the dimension of space $d\geq3$, we just change the integrability condition of $\rho_f$ in Theorem \ref{thm1.1} into $\rho_f \in L^\infty(\mathbb{R}_+;L^{d/2}(\mathbb{T}^3))$.
	\end{rem}
	\begin{rem}
		The global uniform-in-time bound assumption of $\Omega_w(f,t)$ can be replaced by the bounded condition of $u\in L^\infty(\mathbb{R}_+\times\mathbb{T}^3)$. Under the above assumptions, we can show that the estimate of $M_2f$ in Lemma $\ref{lem4.1}$ is indeed independent of time, and applying this time-independent estimate in the estimates in Lemma $\ref{lem4.2}$, we obtain the global uniform in time bound of $\Omega_w(f,t)$.
	\end{rem}
	
	The following theorem provides a condition for the existence of weak solutions without the assumption of small initial values.
	\begin{thm}\label{thm1.2}
		Let $T>0$ be given, and suppose the initial data $[f^\mathrm{in}, u^\mathrm{in}]$ satisfy the following conditions:
		
		$$f^{\mathrm{in}}\in L^\infty(\mathbb{T}^3\times\mathbb{R}^3),\quad\int_{\mathbb{R}^3}|w|^2f^{\mathrm{in}}(x,w)\mathrm{d}w\in L^\infty(\mathbb{T}^3),\quad u^{\mathrm{in}}\in L^2(\mathbb{T}^3).$$
		Then there exists at least one weak solution in the sense of Definition \ref{df4.1} satisfying the following estimates:
		$$\begin{aligned}&(1)\quad\|f\|_{L^\infty((0,T)\times\mathbb{T}^3\times\mathbb{R}^3)}\leq\mathrm{e}^{CT}\|f^{\mathrm{in}}\|_{L^\infty(\mathbb{T}^3\times\mathbb{R}^3)},\\&(2)\quad\frac{1}{2}\int_{\mathbb{T}^3\times\mathbb{R}^3}|w|^2 f(t,x,w) \mathrm{d}x\mathrm{d}w+\frac12\|u(t)\|_{L^2(\mathbb{T}^3)}^2+\int_0^t\|\nabla u(s)\|_{L^2(\mathbb{T}^3)}^2\mathrm{d}s+\int_0^t\int_{\mathbb{T}^3\times\mathbb{R}^3}|u-w|^2f(s,x,w)\mathrm{d}x\mathrm{d}w \mathrm{d}s\\&\quad\leq \frac{1}{2}\int_{\mathbb{T}^3\times\mathbb{R}^3}|w|^2f^{\mathrm{in}}(x,w) \mathrm{d}x\mathrm{d}w+\frac12\|u^{\mathrm{in}}\|_{L^2(\mathbb{T}^3)}^2.\end{aligned}$$
	\end{thm}
	Next, we provide some comments on \ref{thm1.2}.
	\begin{rem}
		Since the relativistic Cucker-Smale equation is also a Vlasov-type equation, this theorem is very similar to other Vlasov-type equations in \cite{ref29,ref30,ref31}. Especially, we draw many inspirations and the method from proving the existence of weak solutions for the thermodynamic Cucker-Smale model coupled with incompressible viscous fluids. 
	\end{rem}

	\section{Preliminaries}
	In this section, we first construct the RCS-NS model using friction theory. Then, we will display some known results. These results will play an important role in the research of the RCS-NS model. Finally, we will give the conservation laws and energy dissipation of RCS-NS model.
	\subsection{Construction of RCS-NS model}
	We again define the map $\hat{w}:B_c(0)\to\mathbb{R}^d$ by
	$$\hat{w}(v):=F(v)v=\Gamma(v)\biggl(1+\frac{\Gamma(v)}{c^2}\biggr)v,\quad\Gamma(v):=\frac{c}{\sqrt{c^2-|v|^2}}.$$
	Then, it was verified in \cite{ref12} that $\hat{w}$ is one-to-one, and it has an inverse function $\hat{v}:\mathbb{R}^d\to B_c(0).$ Furthermore, we consider the map $g:|v|\mapsto|\hat{w}(v)|$ defined by 
	$$g(|v|):=\frac{c|v|}{\sqrt{c^2-|v|^2}}+\frac{|v|}{c^2-|v|^2}$$
	and its inverse $g^{-1}:|w|\mapsto|\hat{v}(w)|.$ As long as there is no confusion, we also suppress “hat” notations in $\hat{v},\hat{w}$ notations and simply write $w$ and $v$ instead of $\hat{w}$ and $\hat{v}$, respectively. We also note that $$w_i:=F(v_i)v_i.$$
	Observe that it is straightforward to see that $g$ and $g^{-1}$ are increasing functions. 
	
	The authors has provided the form of Relativistic Cucker-Smale model in \cite{ref10}, as following,
	\begin{equation}
		\small{	\left\{
			\begin{aligned}
				&\displaystyle{\frac{\mathrm{d}x_i}{\mathrm{d}t}=v_i,\quad i\in \left[ N \right] :=\left\{ 1,2,...,N \right\} ,\quad t>0,}\\
				&\displaystyle{\frac{\mathrm{d}w_i}{\mathrm{d}t}=\frac{1}{N}\sum_{j=1}^N{\phi}\left( |x_i-x_j| \right) \left( v_j-v_i \right).}
			\end{aligned}
			\right.}\nonumber
	\end{equation}
	When the number of particle is to large, it is too complex if we continue using ODE to describe its moving. In a mean-field regime $N\gg1$, the RCS model can be effectively approximated by the corresponding mean-field kinetic equation, namely, the relativistic Kinetic Cucker-Smale (RKCS) equation for one-particle distribution function $f=f(t,x,w)$,
	\begin{equation}
		\small{	\left\{
			\begin{aligned}
				&\partial_tf+\hat{v}(w)\cdot\nabla_xf+\nabla_w\cdot(L[f]f)=0,\\
				&L[f](t,x,w):=-\int_{\mathbb{R}^{2d}}\phi(|x_*-x|)(\hat{v}(w)-\hat{v}(w_*))f(t,x_*,w_*)\mathrm{d}x_*\mathrm{d}w_*,\\
				&f(0,x,w)=f^{\mathrm{in}}(x,w).           
			\end{aligned}
			\right.}\nonumber
	\end{equation}
	The mean-field limit has been proved strictly in \cite{ref11}.
	
	Next, we will consider the fluid part. Because the fluid considered in this study is macroscopic and bulk, we still use the incompressible Navier-Stokes equations without considering relativistic effects. This consideration is relatively rational in physics and can appropriately reduce the difficulty of mathematical processing. The famous incompressible Navier-Stokes equations has the following form,
	$$\begin{cases}\partial_tu+(u\cdot \nabla)u+\nabla p-\mu \Delta u=0,\\ \nabla\cdot u=0.\end{cases}$$
	
	Finally, we will consider the interaction between particles and fluid. From \cite{ref26,ref27}, we know that the interaction is through friction forces. Hence, we shall add
	$$-\int_{\mathbb{R}^3}(u(t,x)-w)f(t,x,w)\mathrm{d}w$$
	on the R.H.S of the momentum part of Navier-Stokes equations, and add
	$$\nabla_w\cdot((u-w)f)$$
	on the L.H.S of the RCS part. So far, we have finished the construction of the RCS-NS model \eqref{eq1}.
	
	\subsection{Known Results}
	In this subsection, we present two known conclusions that will help us prove the existence of weak solutions.
	\begin{prop}\cite{ref29}\label{prop2.1}
		Let $T> 0$ and consider a sequence $\{a_n\}_{n\in \mathbb{N}}$ of nonnegative continuous functions on $[0, T]$. Assume that  $\{a_n\}_{n\in \mathbb{N}}$  satisfies, for any $n$,
		$$a_{n+1}(t)\leq A+B\int_0^ta_n(s)\:\mathrm \mathrm{d}s+C\int_0^ta_{n+1}(s)\:\mathrm \mathrm{d}s,\quad0\leq t\leq T,$$
		where $A$, $B$ and $C$ are nonnegative constants.\\
		If $A= 0$, there exists a constant $K\geq 0$ such that
		$$a_n(t)\leq\dfrac{K^nt^n}{n!},\quad0\leq t\leq T,\quad n\in\mathbb{N}.$$
		If $A>0$, there exists a constant $K \geq 0$ depending on $A$, $B$, $C$ such that
		$$a_{n}(t)\leq K\exp(Kt),\quad0\leq t\leq T,\quad n\in\mathbb{N}.$$
	\end{prop}
	We next give velocity-moment estimates in the lemmas below. We set
	\begin{equation}
		m_{\alpha}f(t,x)= \int_{\mathbb{R}^3}|w|^\alpha f(t,x,w) \mathrm{d}w,  M_{\alpha}f(t)=\int_{\mathbb{T}^3\times\mathbb{R}^3}|w|^\alpha f(t,x,w) \mathrm{d}x\mathrm{d}w
	\end{equation}
	for any $t\in[0,T]$, $ {x}\in\mathbb{T}^{3}$ and $\alpha\geq0$. Note that 
	$$M_{\alpha}f(t)  =\int_{\mathbb{T}^3}m_\alpha f(t, {x})\mathrm{d}{x}.$$
	\begin{prop}\cite{ref29}\label{prop2.2}
		Let $\beta> 0$ and $g$ be a non-negative function in $L^\infty ( ( 0, T) \times \mathbb{T} ^3\times \mathbb{R} ^3)$, such that $m_\beta g( t, x ) < + \infty$, for a.e. $( t, x ) .$ The following estimate holds for any $\alpha < \beta :$
		$$m_\alpha g(t,x)\leq\left(\frac{4}{3}\pi\|g(t,x,\cdot)\|_{L^\infty(\mathbb{R}^3)}+1\right)m_\beta g(t,x)^{\frac{\alpha+3}{\beta+3}},\quad a.e. (t,x).$$
	\end{prop}
	
	\subsection{Some estimates of Lorentz transformation}
	In this subsection, we will give some known propositions of the Lorentz transformation. These propositions show the relationship between the energy function $E$, the velocity $v$ and relativistic velocity $w$. In fact, these results can be reached by direct calculation.
	\begin{prop}\cite{ref32}
		For $v \in B_c(0)$, let $w = \hat{w}(v)$ be a function defined by $\displaystyle{\hat{w}(v):=\Gamma(v)\left(1+\frac{\Gamma(v)}{c^2}\right)v}$ as above. Then, one has
		\begin{equation}
			\nabla_w E=v.
		\end{equation}
	\end{prop}
	\begin{prop}\cite{ref32}\label{prop2.4}
		For $v \in B_c(0)$, let $w = \hat{w}(v)$ be a function defined by $\displaystyle{\hat{w}(v):=\Gamma(v)\left(1+\frac{\Gamma(v)}{c^2}\right)v}$ as above. Then, the Jacobian $\nabla_v \hat{w}$ has two eigenvalues. Furthermore, the eigenvalues of
		$\nabla_v \hat{w}$ are $\lambda_1$ with multiplicity $d-1$ and $\lambda_2$ with multiplicity 1,
		\begin{equation}
			\lambda_1(v)=\frac{c}{\sqrt{c^2-|v|^2}}+\frac{1}{c^2-|v|^2}=F(v)=1+|O(c^{-2})|
		\end{equation}
		and
		\begin{equation}
			\lambda_2(v)=\frac{c|v|^2}{(c^2-|v|^2)^{\frac{3}{2}}}+\frac{2|v|^2}{(c^2-|v|^2)^2}+\frac{c}{\sqrt{c^2-|v|^2}}+\frac{1}{c^2-|v|^2}=1+|O(c^{-2})|.
		\end{equation}
	\end{prop}
	\begin{prop}\cite{ref32}
		For $v \in B_c(0)$, let $w = \hat{w}(v)$ be a function defined by $\displaystyle{\hat{w}(v):=\Gamma(v)\left(1+\frac{\Gamma(v)}{c^2}\right)v}$ as above. Then, one has
		\begin{equation}
			\nabla_{w}\cdot \hat{v}=d-|O(c^{-2})|<d, \quad\det(\nabla_{w} \hat{v})=1-|O(c^{-2})|,\quad\det(\nabla_{v}\hat{w})=1+|O(c^{-2})|
		\end{equation}
	\end{prop}
	\begin{rem}\label{rem2.1}
		From Proposition \ref{prop2.4}, we note that $\lambda_1(v)=:g_1(|v|)$ and $\lambda_2(v)=:g_2(|v|)$. Then, for a constant $V$ such that $V<c$, both $g_1(r)$ and $g_2(r)$ are Lipschitz continuous on the interval $[0,V]$, i.e.
		$$\exists L(V)>0 \text{ s.t. }\forall s,t\in[0,V],\, |g_i(s)-g_i(t)|\leq L(V)|s-t|,\, i=1,2.$$
		Therefore, for $v_1:=\hat{v}(w_1)$, $v_2:=\hat{v}(w_2) \in B_V(0)$, we have
		\begin{align}
			&\nabla_w\cdot v_1-\nabla_w\cdot v_2=\left(\frac{d-1}{g_1(|v_1|)}+\frac{1}{g_2(|v_1|)}\right)-\left(\frac{d-1}{g_1(|v_2|)}+\frac{1}{g_2(|v_2|)}\right)\nonumber\\
			&\leq \left[\left(\frac{d-1}{g_1(r)}+\frac{1}{g_2(r)}\right)\right]_{\text{Lips}}\bigg||v_1|-|v_2|\bigg|\leq \left(\frac{(d-1)L(V)}{[g_1(0)]^2}+\frac{L(V)}{[g_2(0)]^2}\right)|v_1-v_2|=:C(V)|v_1-v_2|.
		\end{align}
		This means $\nabla_w\cdot v$ is Lipschitz continuous in some sense. 
	\end{rem}
	\begin{prop}\cite{ref32}
		For $v_1$, $v_2 \in B_c(0)$, let $w_i := \hat{w}(v_i)$ be a function defined by $\displaystyle{\hat{w}(v):=\Gamma(v)\left(1+\frac{\Gamma(v)}{c^2}\right)v}$ as above. Then, one has
		\begin{equation}
			\frac{c^2+1}{c^2}|v_1-v_2|^2\leq (v_1-v_2)\cdot(w_1-w_2)\leq  \frac{c^2}{c^2+1}|w_1-w_2|^2.
		\end{equation}
	\end{prop}
	
	\subsection{Conservation laws and energy dissipation}
	The following lemmas will describe the change rule of the mass, momentum and energy of the RCS-NS system.
	\begin{lem}\label{lem2.1}
		Let $[f (t, x, w),u(t,x)]$ be a smooth solution to \eqref{eq1}, and let $f(t,x,w)$ decay to zero sufficiently fast in the $w$-variable. Then, we have the following assertions.
		
		(1) The conservation of mass: $$\displaystyle{\frac{\mathrm{d}}{\mathrm{d}t}\int_{\mathbb{T}^3 \times \mathbb{R}^3}f(t,x,w)\mathrm{d}x\mathrm{d}w=0},$$
		
		(2) The conservation of momentum: $$\displaystyle{\frac{\mathrm{d}}{\mathrm{d}t}\left(\int_{\mathbb{T}^3 \times \mathbb{R}^3}w f(t,x,w)\mathrm{d}x\mathrm{d}w+\int_{\mathbb{T}^3}u\mathrm{d}x\right)=0,}$$
		
		(3) The estimate of total energy: 
		\begin{align}\label{eq2.8}
			&\frac{\mathrm{d}}{\mathrm{d}t}\left(\int_{\mathbb{T}^3 \times \mathbb{R}^3}E f(t,x,w)\mathrm{d}x\mathrm{d}w+\frac{1}{2}\int_{\mathbb{T}^3}|u|^2\mathrm{d}x\right)+\mu \int_{\mathbb{T}^3}|\nabla u|^2\mathrm{d}x+\int_{\mathbb{T}^3 \times \mathbb{R}^3}(u-\hat{v}(w))\cdot (u-w)f\mathrm{d}x\mathrm{d}w\nonumber\\
			&=-\frac{1}{2}\int_{\mathbb{T}^6 \times \mathbb{R}^6}\phi(|x-x_*|)|\hat{v}(w_*)-\hat{v}(w)|^2f(t,x,w))f(t,x_*,w_*)\mathrm{d}x\mathrm{d}w \mathrm{d}x_*\mathrm{d}w_*.
		\end{align}
	\end{lem}
	\begin{proof}
		(1) For the conservation of mass, we integrate the first equation \eqref{eq1} over $\mathbb{T}^3 \times \mathbb{R}^3$ to obtain
		\begin{equation}
			\frac{\mathrm{d}}{\mathrm{d}t}\int_{\mathbb{T}^3 \times \mathbb{R}^3}f(t,x,w)\mathrm{d}x\mathrm{d}w=\int_{\mathbb{T}^3 \times \mathbb{R}^3}f(t,x,w)\nabla_{x}\cdot \hat{v}(w) \mathrm{d}x\mathrm{d}w=0
		\end{equation}
		(2) For the conservation of momentum, we multiply the first equation \eqref{eq1} by $w$ and integrate the resulting equation over $\mathbb{T}^3 \times \mathbb{R}^3$ to obtain
		\begin{equation}\label{eq2}
			\frac{\mathrm{d}}{\mathrm{d}t}\int_{\mathbb{T}^3 \times \mathbb{R}^3}w f(t,x,w)\mathrm{d}x\mathrm{d}w=-\int_{\mathbb{T}^3 \times \mathbb{R}^3}w \nabla_{w}\cdot [L[f]+F_d)f] \mathrm{d}x\mathrm{d}w=\int_{\mathbb{T}^3 \times \mathbb{R}^3}(L[f]+F_d)f\mathrm{d}x\mathrm{d}w.
		\end{equation}
		Since
		\begin{align}
			\int_{\mathbb{T}^3 \times \mathbb{R}^3}L[f]f\mathrm{d}x\mathrm{d}w&=\int_{\mathbb{T}^6 \times \mathbb{R}^6}\phi(|x-x_*|)(\hat{v}(w_*)-\hat{v}(w))f(t,x_*,w_*)f(t,x,w)\mathrm{d}x_*\mathrm{d}w_*\mathrm{d}x\mathrm{d}w\nonumber\\
			&=\int_{\mathbb{T}^6 \times \mathbb{R}^6}\phi(|x-x_*|)(\hat{v}(w)-\hat{v}(w_*))f(t,x,w)f(t,x_*,w_*)\mathrm{d}x\mathrm{d}w \mathrm{d}x_*\mathrm{d}w_*\nonumber\\
			&\quad \text{by } (x,w)\leftrightarrow (x_*,w_*)\nonumber\\
			&=-\int_{\mathbb{T}^3 \times \mathbb{R}^3}L[f]f\mathrm{d}x\mathrm{d}w,
		\end{align}
		this implies
		\begin{equation}
			\int_{\mathbb{T}^3 \times \mathbb{R}^3}L[f]f\mathrm{d}x\mathrm{d}w=0.
		\end{equation}
		Therefore, it follows from \eqref{eq2} that
		\begin{equation}\label{eq3}
			\frac{\mathrm{d}}{\mathrm{d}t}\int_{\mathbb{T}^3 \times \mathbb{R}^3}w f(t,x,w)\mathrm{d}x\mathrm{d}w=\int_{\mathbb{T}^3 \times \mathbb{R}^3}F_df\mathrm{d}x\mathrm{d}w=\int_{\mathbb{T}^3 \times \mathbb{R}^3}(u(t,x)-w)f\mathrm{d}x\mathrm{d}w.
		\end{equation}
		On the other hand, it follows from the second equation in \eqref{eq1} that
		\begin{equation}\label{eq4}
			\frac{\mathrm{d}}{\mathrm{d}t}\int_{\mathbb{T}^3}u\mathrm{d}x=-\int_{\mathbb{T}^3 \times \mathbb{R}^3}(u(t,x)-w)f\mathrm{d}x\mathrm{d}w.
		\end{equation}
		Hence, we deduce from \eqref{eq3} and \eqref{eq4} to obtain the conservation of momentum.
		
		(3)For the estimate of energy, we multiply the first equation \eqref{eq1} by $E(v)$ and integrate the resulting equation over $\mathbb{T}^3 \times \mathbb{R}^3$ to obtain
		\begin{align}
			\frac{\mathrm{d}}{\mathrm{d}t}\int_{\mathbb{T}^3 \times \mathbb{R}^3}E f(t,x,w)\mathrm{d}x\mathrm{d}w&=-\int_{\mathbb{T}^3 \times \mathbb{R}^3}E \nabla_{w}\cdot [L[f]+F_d)f] \mathrm{d}x\mathrm{d}w=\int_{\mathbb{T}^3 \times \mathbb{R}^3}\nabla_{w} E \cdot(L[f]+F_d)f\mathrm{d}x\mathrm{d}w\nonumber\\
			&=\int_{\mathbb{T}^3 \times \mathbb{R}^3}\hat{v}(w)\cdot(L[f]+F_d)f\mathrm{d}x\mathrm{d}w.
		\end{align}
		Since
		\begin{align}
			\int_{\mathbb{T}^3 \times \mathbb{R}^3}\hat{v}(w)\cdot L[f]f\mathrm{d}x\mathrm{d}w&=\int_{\mathbb{T}^6 \times \mathbb{R}^6}\phi(|x-x_*|)\hat{v}(w)\cdot(\hat{v}(w_*)-\hat{v}(w))f(t,x_*,w_*)f(t,x,w)\mathrm{d}x_*\mathrm{d}w_*\mathrm{d}x\mathrm{d}w\nonumber\\
			&=-\int_{\mathbb{T}^6 \times \mathbb{R}^6}\phi(|x-x_*|)\hat{v}(w_*)\cdot(\hat{v}(w_*)-\hat{v}(w)f(t,x,w))f(t,x_*,w_*)\mathrm{d}x\mathrm{d}w \mathrm{d}x_*\mathrm{d}w_*\nonumber\\
			&\quad \text{by } (x,w)\leftrightarrow (x_*,w_*)\nonumber\\
			&=-\frac{1}{2}\int_{\mathbb{T}^6 \times \mathbb{R}^6}\phi(|x-x_*|)|\hat{v}(w_*)-\hat{v}(w)|^2f(t,x,w))f(t,x_*,w_*)\mathrm{d}x\mathrm{d}w \mathrm{d}x_*\mathrm{d}w_*,
		\end{align}
		then we have
		\begin{align}\label{eq5}
			\frac{\mathrm{d}}{\mathrm{d}t}\int_{\mathbb{T}^3 \times \mathbb{R}^3}E f(t,x,w)\mathrm{d}x\mathrm{d}w&=\int_{\mathbb{T}^3 \times \mathbb{R}^3}\hat{v}(w)\cdot(L[f]+F_d)f\mathrm{d}x\mathrm{d}w\nonumber\\
			&=-\frac{1}{2}\int_{\mathbb{T}^6 \times \mathbb{R}^6}\phi(|x-x_*|)|\hat{v}(w_*)-\hat{v}(w)|^2f(t,x,w))f(t,x_*,w_*)\mathrm{d}x\mathrm{d}w \mathrm{d}x_*\mathrm{d}w_*\nonumber\\
			&\quad+\int_{\mathbb{T}^3 \times \mathbb{R}^3}\hat{v}(w)\cdot(u(t,x)-w)f\mathrm{d}x\mathrm{d}w.
		\end{align}
		On the other hand, it follows from the second equation in \eqref{eq1} that
		\begin{equation}\label{eq6}
			\frac{1}{2}\frac{\mathrm{d}}{\mathrm{d}t}\int_{\mathbb{T}^3}|u|^2\mathrm{d}x+\mu \int_{\mathbb{T}^3}|\nabla u|^2\mathrm{d}x=-\int_{\mathbb{T}^3 \times \mathbb{R}^3}u\cdot(u-w)f\mathrm{d}x\mathrm{d}w.
		\end{equation}
		Hence, we deduce from \eqref{eq5} and \eqref{eq6} the dissipative structure of the energy.
	\end{proof}
	\begin{rem}
		Unlike non-relativistic situations, the energy estimate \eqref{eq2.8} does not imply the monotonicity of the total energy as it is. More precisely, it is not clear to get that the total energy is not increasing.
	\end{rem}
	\section{Asymptomatic flocking estimate}
	If the fluid velocity is larger than the particle velocity, i.e. $u(x,t) > w$ , then $\displaystyle{-\int_{\mathbb{T}^3\times\mathbb{R}^3}F_df \mathrm{d}x_*\mathrm{d}w_*}$
	is negative; hence, the fluid will decelerate to match the particle's relativistic velocity. In contrast, if the fluid velocity is smaller than the relativistic velocity of the particle, then the drag force makes the fluid accelerate, so alignment is achieved. Therefore, we can view the drag force as a kind of an alignment force.
	
	By conducting a detailed analysis of the Lyapunov functional, we demonstrate the exponential stability of the RCS-NS system in this subsection.
	\begin{lem}[Time-variation estimate for $\mathcal{E}^w$]\label{lem3.1}
		Let $[f (t, x, w),u(t,x)]$ be a smooth solution to \eqref{eq1} satisfying 
		\begin{equation}
			\int_{\mathbb{T}^3 \times \mathbb{R}^3}f^{\text{in}}(x,w)\mathrm{d}x\mathrm{d}w=1,\quad \bigcup\limits_{t\geq 0} \Omega_w (f,t)\subset B_{W_0}(0)\quad \text{for some} \quad W_0>0,
		\end{equation}
		and $f(t,x,w)$ decays to zero sufficiently fast in the $w$-variable. Then, we have
		\begin{equation}
			\frac{\mathrm{d}\mathcal{E}^w(t)}{\mathrm{d}t}\leq-2C(c,W_0)\phi(\sqrt{3})\mathcal{E}^w(t)+2\int_{\mathbb{T}^3 \times \mathbb{R}^3}(w-w_c(t))\cdot (u-w)f\mathrm{d}x\mathrm{d}w.
		\end{equation}
	\end{lem}
	\begin{proof}
		By direct calculation, we have
		\begin{align}
			\frac{\mathrm{d}\mathcal{E}^w(t)}{\mathrm{d}t}&=-2w_c(t)\cdot\int_{\mathbb{T}^3 \times \mathbb{R}^3}(w-w_c(t))f(t,x,w)\mathrm{d}x\mathrm{d}w+\int_{\mathbb{T}^3 \times \mathbb{R}^3}|w-w_c(t)|^2\partial_tf(t,x,w)\mathrm{d}x\mathrm{d}w\nonumber\\
			&=:\mathcal{I}_1(t)+\mathcal{I}_2(t).
		\end{align}
		To estimate $\mathcal{I}_1$, we use the definition of $w_c(t)$ and unit mass assumption to obtain 
		\begin{equation}
			\int_{\mathbb{T}^3 \times \mathbb{R}^3}(w-w_c(t))f(t,x,w)\mathrm{d}x\mathrm{d}w=0.
		\end{equation}
		Therefore, we have $\mathcal{I}_1(t)=0$.
		
		To estimate $\mathcal{I}_2$, we use
		\begin{equation}
			\partial_t f=-\hat{v}(w) \cdot \nabla_x f-\nabla_w \cdot [(L[f]+F_d)f],
		\end{equation}
		and integration by parts to find
		\begin{align}
			\mathcal{I}_2(t)&=\int_{\mathbb{T}^3 \times \mathbb{R}^3}|w-w_c(t)|^2\partial_tf\mathrm{d}x\mathrm{d}w\nonumber\\
			&=2\int_{\mathbb{T}^3 \times \mathbb{R}^3}(w-w_c(t))\cdot [(L[f]+F_d)f]\mathrm{d}x\mathrm{d}w\nonumber\\
			&=2\int_{\mathbb{T}^6 \times \mathbb{R}^6}\phi(|x-x_*|)(w-w_c(t))\cdot (\hat{v}(w_*)-\hat{v}(w))f(t,x_*,w_*)f(t,x,w)\mathrm{d}x\mathrm{d}w \mathrm{d}x_*\mathrm{d}w_*\nonumber\\
			&\quad +2\int_{\mathbb{T}^3 \times \mathbb{R}^3}(w-w_c(t))\cdot (u-w)f\mathrm{d}x\mathrm{d}w  \nonumber\\
			&=:\mathcal{I}_{21}(t)+\mathcal{I}_{22}(t)
		\end{align}
		Since for $\forall w_1, w_2 \in B_{W_0}(0)$, the map $\hat{w}:B_c(0)\rightarrow \mathbb{R}^n$ satisfies
		\begin{equation}
			|w_1-w_2|=|\hat{w}(v_1)-\hat{w}(v_2)|\leq |\nabla_v \hat{w}|_{\text{op}}|v_1-v_2|
		\end{equation}
		where $|\nabla_v \hat{w}|_{\text{op}}$ denotes the operator norm of the matrix $\nabla_v \hat{w}$. Since $\nabla_v \hat{w}$ is asymmetric positive definite matrix,the operator norm is
		nothing but its largest eigenvalue, which is $\lambda_2(v)$ in our case. Therefore, one has
		\begin{equation}
			|w_1-w_2|=|\hat{w}(v_1)-\hat{w}(v_2)|\leq |\nabla_v \hat{w}|_{\text{op}}|v_1-v_2|\leq \left(\sup\limits_{v\in B_{\hat{v}(W_0)}(0)}\lambda_2(v)\right)|v_1-v_2|=:C_0(c,W_0)|v_1-v_2|.
		\end{equation}
		The first term $\mathcal{I}_{21}$ can be preliminarily estimated as follows:
		\begin{align}
			\mathcal{I}_{21}(t)&=2\int_{\mathbb{T}^6 \times \mathbb{R}^6}\phi(|x-x_*|)(w-w_c(t))\cdot (\hat{v}(w_*)-\hat{v}(w))f(t,x_*,w_*)f(t,x,w)\mathrm{d}x\mathrm{d}w \mathrm{d}x_*\mathrm{d}w_*\nonumber\\
			&=2\int_{\mathbb{T}^6 \times \mathbb{R}^6}\phi(|x-x_*|)w\cdot (\hat{v}(w_*)-\hat{v}(w))f(t,x_*,w_*)f(t,x,w)\mathrm{d}x\mathrm{d}w \mathrm{d}x_*\mathrm{d}w_*\nonumber\\
			&=-2\int_{\mathbb{T}^6 \times \mathbb{R}^6}\phi(|x-x_*|)w_*\cdot (\hat{v}(w_*)-\hat{v}(w))f(t,x_*,w_*)f(t,x,w)\mathrm{d}x\mathrm{d}w \mathrm{d}x_*\mathrm{d}w_*\nonumber\\
			&\quad \text{by } (x,w)\leftrightarrow (x_*,w_*)\nonumber\\
			&=-\int_{\mathbb{T}^6 \times \mathbb{R}^6}\phi(|x-x_*|)(w_*-w)\cdot (\hat{v}(w_*)-\hat{v}(w))f(t,x_*,w_*)f(t,x,w)\mathrm{d}x\mathrm{d}w \mathrm{d}x_*\mathrm{d}w_*\nonumber\\
			&\leq -\phi(\sqrt{3})\int_{\mathbb{T}^6 \times \mathbb{R}^6}\frac{c^2+1}{c^2}|\hat{v}(w_*)-\hat{v}(w)|^2f(t,x_*,w_*)f(t,x,w)\mathrm{d}x\mathrm{d}w \mathrm{d}x_*\mathrm{d}w_*\nonumber\\
			&\leq -\phi(\sqrt{3})\int_{\mathbb{T}^6 \times \mathbb{R}^6}\frac{c^2+1}{c^2[C_0(c,W_0)]^2}|w_*-w|^2f(t,x_*,w_*)f(t,x,w)\mathrm{d}x\mathrm{d}w \mathrm{d}x_*\mathrm{d}w_*
		\end{align}
		By applying the definition of $w_c(t)$, we have
		\begin{equation}
			\int_{\mathbb{T}^6 \times \mathbb{R}^6}(w_*-w_c)\cdot(w-w_c)f(t,x_*,w_*)f(t,x,w)\mathrm{d}x\mathrm{d}w \mathrm{d}x_*\mathrm{d}w_*=\left|\int_{\mathbb{T}^3\times\mathbb{R}^3}(w-w_c)f(t,x,w)\mathrm{d}x\mathrm{d}w\right|^2=0.
		\end{equation}
		Then, by applying the triangle equality
		$|w_*-w|^2=|w_*-w_c|^2+|w-w_c|^2+2(w_*-w_c)\cdot(w-w_c)$, the first term $\mathcal{I}_{21}$ can be finally estimated as follows:
		\begin{align}\label{eq7}
			\mathcal{I}_{21}(t)&\leq -\phi(\sqrt{3})\int_{\mathbb{T}^6 \times \mathbb{R}^6}\frac{c^2+1}{c^2[C_0(c,W_0)]^2}(|w_*-w_c|^2+|w-w_c|^2)f(t,x_*,w_*)f(t,x,w)\mathrm{d}x\mathrm{d}w \mathrm{d}x_*\mathrm{d}w_*\nonumber\\
			&=-\frac{2(c^2+1)}{c^2[C_0(c,W_0)]^2}\phi(\sqrt{3})\mathcal{E}^w(t)=:-2C(c,W_0)\phi(\sqrt{3})\mathcal{E}^w(t).
		\end{align}
		
		On the other hand, we split term $\mathcal{I}_{22}$ as two terms:
		\begin{equation}\label{eq8}
			\mathcal{I}_{22}=2\int_{\mathbb{T}^3 \times \mathbb{R}^3}(w-w_c(t))\cdot (u-w)f\mathrm{d}x\mathrm{d}w=2\int_{\mathbb{T}^3 \times \mathbb{R}^3}w\cdot (u-w)f\mathrm{d}x\mathrm{d}w-2\int_{\mathbb{T}^3 \times \mathbb{R}^3}w_c(t)\cdot (u-w)f\mathrm{d}x\mathrm{d}w.
		\end{equation}
		
		We finally combine \eqref{eq7} and \eqref{eq8} to obtain the desired result.
	\end{proof}
	\begin{lem}[Time-variation estimate for $\mathcal{L}$]\label{lem3.2}
		Let $[f (t, x, w),u(t,x)]$ be a smooth solution to \eqref{eq1} satisfying 
		\begin{equation}
			\int_{\mathbb{T}^3 \times \mathbb{R}^3}f^{\text{in}}(x,w)\mathrm{d}x\mathrm{d}w=1,\quad \bigcup\limits_{t\geq 0} \Omega_w (f,t)\subset B_{W_0}(0)\quad \text{for some} \quad W_0>0,
		\end{equation}
		and $f(t,x,w)$ decays to zero sufficiently fast in the $w$-variable. Then, we have
		\begin{equation}
			\frac{\mathrm{d}}{\mathrm{d}t}\mathcal{L}(t)+2\mathcal{D}(t) \leq 0 
		\end{equation}
		where $\mathcal{D}(t)$ is given by
		\begin{equation}
			\mathcal{D}(t):=C(c,W_0)\phi(\sqrt{3})\mathcal{E}^w(t)+\mu\int_{\mathbb{T}^3}|\nabla u|^2\mathrm{d}x+\int_{\mathbb{T}^3 \times \mathbb{R}^3}|u-w|^2f\mathrm{d}x\mathrm{d}w. 
		\end{equation}
	\end{lem}
	\begin{proof}
		By direct calculation, we have
		\begin{equation}
			\frac{\mathrm{d}\mathcal{E}^u}{\mathrm{d}t}=2\int_{\mathbb{T}^3}(u-u_c)\cdot u_t\mathrm{d}x-2\int_{\mathbb{T}^3}(u-u_c)\cdot u_c'\mathrm{d}x=:2(\mathcal{J}_1+\mathcal{J}_2).
		\end{equation}
		By using the Navier-Stokes part of \eqref{eq1}, $\mathcal{J}_1$ can be calculated as 
		\begin{align}\label{eq9}
			\mathcal{J}_1&=-\int_{\mathbb{T}^3}(u\cdot \nabla)u \cdot(u-u_c) \mathrm{d}x-\int_{\mathbb{T}^3}\nabla p \cdot(u-u_c) \mathrm{d}x+\mu \int_{\mathbb{T}^3}\Delta u \cdot(u-u_c) \mathrm{d}x-\int_{\mathbb{T}^3\times \mathbb{R}^3}(u-w) \cdot(u-u_c) \mathrm{d}x\mathrm{d}w \nonumber\\
			&=:\mathcal{J}_{11}+\mathcal{J}_{12}+\mathcal{J}_{13}+\mathcal{J}_{14}=0+0-\mu\int_{\mathbb{T}^3}|\nabla u|^2\mathrm{d}x-\int_{\mathbb{T}^3\times \mathbb{R}^3}(u-w) \cdot(u-u_c)f \mathrm{d}x\mathrm{d}w.
		\end{align}
		By using the definition of $u_c$, $\mathcal{J}_2$ can be calculated as 
		\begin{equation}\label{eq10}
			\mathcal{J}_2=-2u_c' \cdot \int_{\mathbb{T}^3}(u-u_c)\mathrm{d}x=0.
		\end{equation}
		Therefore, we have
		\begin{equation}
			\frac{\mathrm{d}\mathcal{E}^u}{\mathrm{d}t}=-2\mu\int_{\mathbb{T}^3}|\nabla u|^2\mathrm{d}x-2\int_{\mathbb{T}^3\times \mathbb{R}^3}(u-w) \cdot(u-u_c)f \mathrm{d}x\mathrm{d}w.
		\end{equation}
		On the other hand, it follows from the conservation of momentum law in Lemma \ref{lem2.1} that $u_c=-w_c$, so we have
		\begin{equation}
			\frac{\mathrm{d}\mathcal{E}^r}{\mathrm{d}t}=(u_c-w_c)\cdot\frac{\mathrm{d}}{\mathrm{d}t}(u_c-w_c)=-2(u_c-w_c)\cdot \frac{\mathrm{d}w_c}{\mathrm{d}t}=-2\int_{\mathbb{T}^3\times \mathbb{R}^3}(u_c-w_c)\cdot (u-w)f\mathrm{d}x\mathrm{d}w
		\end{equation}
		
		Finally, by recalling the estimate of $\mathcal{E}^w$ in Lemma \ref{lem3.1}, we obtain
		\begin{equation}
			\frac{\mathrm{d}}{\mathrm{d}t}\mathcal{L}(t)+2C(c,W_0)\phi(\sqrt{3})\mathcal{E}^w(t)+2\mu\int_{\mathbb{T}^3}|\nabla u|^2\mathrm{d}x+2\int_{\mathbb{T}^3 \times \mathbb{R}^3}|u-w|^2f\mathrm{d}x\mathrm{d}w \leq 0.
		\end{equation}
		We reach the desired result.
	\end{proof}
	\begin{lem}[Time-variation estimate for $\mathcal{D}$]\label{lem3.3}
		Let $[f (t, x, w),u(t,x)]$ be a smooth solution to \eqref{eq1} satisfying
		
		(1)\quad $\displaystyle{\int_{\mathbb{T}^3 \times \mathbb{R}^3}f^{\text{in}}(x,w)\mathrm{d}x\mathrm{d}w=1,\quad \bigcup\limits_{t\geq 0} \Omega_w (f,t)\subset B_{W_0}(0)\quad \text{for some} \quad W_0>0}$.
		
		(2)\quad $\displaystyle{||\rho_f||_{L^\infty\left(\mathbb{R}_+\times \mathbb{T}^3\right)}<\infty,\quad \text{where} \quad \rho_f(t,x):=\int_{\mathbb{R}^3}f(t,x,w)\mathrm{d}w}$ and $f(t,x,w)$ decay to zero sufficiently fast in the $w$-variable. Then, we have
		\begin{equation}
			\frac{1}{2}\mathcal{L}(t)\leq C(\rho_f,\mu)\mathcal{D}(t) 
		\end{equation}
	\end{lem}
	\begin{proof}
		By direct calculation and using Young's inequality, for $\forall \varepsilon>0$, we have
		\begin{align}
			&\int_{\mathbb{T}^3 \times \mathbb{R}^3}|u-w|^2f\mathrm{d}x\mathrm{d}w
			=\int_{\mathbb{T}^3 \times \mathbb{R}^3}|u-u_c+u_c-w_c+w_c-w|^2f\mathrm{d}x\mathrm{d}w\nonumber\\
			&=|u_c-w_c|^2+\int_{\mathbb{T}^3 \times \mathbb{R}^3}|u-u_c|^2f\mathrm{d}x\mathrm{d}w+\int_{\mathbb{T}^3 \times \mathbb{R}^3}|w_c-w|^2f\mathrm{d}x\mathrm{d}w+2\int_{\mathbb{T}^3 \times \mathbb{R}^3}(u-u_c)\cdot(u_c-w)f\mathrm{d}x\mathrm{d}w\nonumber\\
			&\geq |u_c-w_c|^2+\left(1-\frac{1}{\varepsilon}\right)\int_{\mathbb{T}^3 \times \mathbb{R}^3}|u-u_c|^2f\mathrm{d}x\mathrm{d}w+\int_{\mathbb{T}^3 \times \mathbb{R}^3}|w_c-w|^2f\mathrm{d}x\mathrm{d}w-\varepsilon\int_{\mathbb{T}^3 \times \mathbb{R}^3}|u_c-w|^2f \mathrm{d}x\mathrm{d}w\nonumber\\
			&=(1-\varepsilon)|u_c-w_c|^2+\left(1-\frac{1}{\varepsilon}\right)\int_{\mathbb{T}^3 \times \mathbb{R}^3}|u-u_c|^2f\mathrm{d}x\mathrm{d}w+(1-\varepsilon)\int_{\mathbb{T}^3 \times \mathbb{R}^3}|w_c-w|^2f \mathrm{d}x\mathrm{d}w.
		\end{align}
		We set $\displaystyle{\varepsilon=\frac{1}{2}}$ to obtain
		\begin{equation}
			\frac{1}{2}|u_c-w_c|^2+\frac{1}{2}\int_{\mathbb{T}^3 \times \mathbb{R}^3}|w_c-w|^2f \mathrm{d}x\mathrm{d}w\leq \int_{\mathbb{T}^3}\rho_f(t,x)|u-u_c|^2\mathrm{d}x+\int_{\mathbb{T}^3 \times \mathbb{R}^3}|u-w|^2f\mathrm{d}x\mathrm{d}w.
		\end{equation}
		Subsequently, we have
		\begin{align}
			\frac{1}{2}\mathcal{L}(t)&\leq \int_{\mathbb{T}^3}(\frac{1}{2}+\rho_f)|u-u_c|^2\mathrm{d}x+\int_{\mathbb{T}^3 \times \mathbb{R}^3}|u-w|^2f\mathrm{d}x\mathrm{d}w+C(c,W_0)\phi(\sqrt{3})\mathcal{E}^w(t)\nonumber\\
			&\leq \left(\frac{1}{2}+||\rho_f||_{L^\infty\left(\mathbb{R}_+\times \mathbb{T}^3\right)}\right)\int_{\mathbb{T}^3}|u-u_c|^2\mathrm{d}x+\int_{\mathbb{T}^3 \times \mathbb{R}^3}|u-w|^2f\mathrm{d}x\mathrm{d}w+C(c,W_0)\phi(\sqrt{3})\mathcal{E}^w(t)\nonumber\\
			&\leq C(\rho_f)\int_{\mathbb{T}^3}|\nabla u|^2\mathrm{d}x+\int_{\mathbb{T}^3 \times \mathbb{R}^3}|u-w|^2f\mathrm{d}x\mathrm{d}w+C(c,W_0)\phi(\sqrt{3})\mathcal{E}^w(t)\nonumber\\
			&\leq \max\left\{1,\frac{C(\rho_f)}{\mu}\right\}\mathcal{D}(t)=:C(\rho_f,\mu)\mathcal{D}(t),
		\end{align}
		where we have used the $Poincar\acute{e}'s$ inequality
		\begin{equation}
			\int_{\mathbb{T}^3}|u-u_c|^2\mathrm{d}x\leq C_P\int_{\mathbb{T}^3}|\nabla u|^2\mathrm{d}x.
		\end{equation}
		Here $C_P$ is the $Poincar\acute{e}'s$ constant.
	\end{proof}
	\begin{prof1}
		By recalling Lemma \ref{lem3.2} and Lemma \ref{lem3.3}, we have
		\begin{equation}
			\frac{\mathrm{d}}{\mathrm{d}t}\mathcal{L}(t)+2\mathcal{D}(t) \leq 0, \quad 
			\frac{1}{2}\mathcal{L}(t)\leq C(\rho_f,\mu)\mathcal{D}(t) .
		\end{equation}
		Therefore, we can easily obtain that
		\begin{equation}
			\frac{\mathrm{d}}{\mathrm{d}t}\mathcal{L}(t)+\frac{1}{C(\rho_f,\mu)}\mathcal{L}(t)\leq \frac{\mathrm{d}}{\mathrm{d}t}\mathcal{L}(t)+2\mathcal{D}(t) \leq 0.
		\end{equation}
		By applying $Gr\ddot{o}nwall's$ inequality, we reach the conclusion
		\begin{equation}
			\mathcal{L}(t)\leq \mathcal{L}(0)\exp\{-\frac{1}{C(\rho_f,\mu)}t\}.
		\end{equation}
	\end{prof1}
	\begin{rem}
		From Theorem \ref{thm1.1}, we know that $\mathcal{E}^w(t)\leq \mathcal{L}(t)\rightarrow 0$. This denotes the formation of velocity alignment in probability. It can be seen from the Chebyshev inequality as follows. For any $\sigma>0$, we have
		$$\begin{aligned}
			\mathcal{E}^w(t) &=\int_{\mathbb{T}^{3}\times\mathbb{R}^{3}}f|w_{c}(t)-w|^{2}\mathrm{d}x\mathrm{d}w\geq\int_{\{|w_{c}(t)-w|>\sigma\}}f|w_{c}(t)-w|^{2}\mathrm{d}x\mathrm{d}w \\
			& \geq\sigma^{2}\int_{\{|w_{c}(t)-w|>\sigma\}}f\mathrm{d}x\mathrm{d}w=\sigma^{2}\mathbb{P}[|w_{c}(t)-w|>\sigma]. \end{aligned}$$
		This implies the weak flocking estimate,
		$$\lim_{t\to\infty}\mathbb{P}[|w_c(t)-w|>\sigma]\leq\frac{1}{\sigma^2}\lim_{t\to\infty}\mathcal{E}^w(t)=0.$$
	\end{rem}
	\section{A global existence of weak solutions}
	In this section, we will construct the form of weak solutions for the RCS-NS system and prove their existence. Fortunately, Reference \cite{ref29} has already provided the form of weak solutions for the coupling structure of Vlasov-type equations and incompressible Navier-Stokes equations and provided ideas and techniques for proving the existence of weak solutions. Therefore, the work in this section will draw on the ideas of Reference \cite{ref29} and, if necessary, directly use the estimates from \cite{ref29}. We hope that the readers have some understanding of the content of \cite{ref29} before reading this section.
	
	First, we introduce some function spaces as follows:
	\begin{equation}
		\mathcal{H}:=\{w\in L^2(\mathbb{T}^3)\mid\nabla\cdot w=0\},\quad\mathcal{V}:=\{w\in H^1(\mathbb{T}^3)\mid\nabla\cdot w=0\},\nonumber
	\end{equation}
	and we denote the dual space of $\mathcal{V}$ by $\mathcal{V}^\prime.$ Next, without loss of generality, we set the viscosity $\mu=1$  and define the concept of the weak solutions to system \eqref{eq1} as follows.\\
	\begin{df}\label{df4.1}
		For a given $T\in(0,\infty)$, the pair $(f,u)$ is a weak solution of \eqref{eq1} on the time interval $[0,T)$ if and only if the pair satisfies the following conditions:\\
		(1) $f\in L^{\infty }( 0, T; ( L^{\infty }\cap L_{+ }^{1}) ( \mathbb{T} ^{3}\times \mathbb{R} ^{3}) )$, $| w | ^{2}f\in L^{\infty }( 0, T; L^{1}( \mathbb{T} ^{3}\times \mathbb{R} ^{3}) ) .$\\
		(2) $u\in L^{\infty }( 0, T; \mathcal{H} ) \cap L^{2}( 0, T; \mathcal{V} ) \cap \mathcal{C} ^{0}( [ 0, T] , \mathcal{V} ^{\prime }) .$\\
		(3) For all $\varphi\in\mathcal{C}^1([0,T]\times\mathbb{T}^3\times\mathbb{R}^3)$ with compact support in $w$, such that
		$\varphi(T,\cdot,\cdot)=0$,
		$$-\int_0^T\int_{\mathbb{T}^3\times\mathbb{R}^3}f\left(\partial_t\varphi+\hat{v}(w)\cdot\nabla\varphi+F[f,u]\cdot\nabla_w\varphi\right)\mathrm{d}x\mathrm{d}w\mathrm{d}s=\int_{\mathbb{T}^3\times\mathbb{R}^3}f^\mathrm{in}\varphi(0,\cdot,\cdot)\:\mathrm{d}x\mathrm{d}w.$$\\
		(4) For all $\psi\in\mathcal{C}^1([0,T]\times\mathbb{T}^3)$ such that $\nabla\cdot\psi=0$,for a.e. $t \in [0,T]$,
		$$\begin{aligned}\int_{\mathbb{T}^3}u(t)\cdot\psi(t)\mathrm{d}x&+\int_0^t\int_{\mathbb{T}^3}\left(-u\cdot\partial_t\psi+(u\cdot\nabla)u\cdot\psi+\nabla u:\nabla\psi\right)\mathrm{d}x\mathrm{d}s\\&=-\int_0^t\int_{\mathbb{T}^3\times\mathbb{R}^3}(u-w)\cdot\psi f\:\mathrm{d}x\mathrm{d}w\mathrm{d}s+\int_{\mathbb{T}^3}u^\mathrm{in}\cdot\psi(\cdot,0)\:\mathrm{d}x.\end{aligned}$$
	\end{df}
	Now, we are ready to state our first main result on the global existence of weak solutions as follows.
	\subsection{A brief introduction to a regularized system}
	Let $\varepsilon>0$ and $m$ be a standard mollifier:
	\begin{equation}
		m \geq 0, \quad m \in C_c^\infty(\mathbb{R}^3) ,\quad \mathrm{supp~}m\subset B_1(0),\quad \int_{\mathbb{R}^3}m(x)\mathrm{d}x=1.
	\end{equation}
	The existence of $m$ can be ensured. For example, $m$ can be chosen as 
	\begin{align}
		m\left( x \right) =\begin{cases}
			C\exp \left( \frac{1}{\left| x \right|^2-1} \right) ,  \text{ if}\,\,\left| x \right|<1\\
			0, \text{ if}\,\,\left| x \right|\ge 1\\
		\end{cases},\quad \text{ where }C:=\left(\int_{B_1(0)}\exp \left( \frac{1}{\left| x \right|^2-1} \right) \mathrm{d}x\right)^{-1}.
	\end{align}
	and we set a family of smooth mollifiers:
	\begin{equation}
		m_\varepsilon(x):=\frac{1}{\varepsilon^3}m(\frac{x}{\varepsilon}),\;\varepsilon>0.
	\end{equation}
	Furthermore, we also introduce a family of smooth cut-off functions $\gamma_\varepsilon(x):$
	\begin{equation}
		\mathrm{supp~}\gamma_{\varepsilon}\subset B_{{\frac{1}{\varepsilon}}}(0),\quad0\leq \gamma_{\varepsilon}\leq 1,\quad\gamma_{\varepsilon}=1\mathrm{~on~}B_{{\frac{1}{2\varepsilon}}}(0),\quad\gamma_{\varepsilon}\to1\mathrm{~as~}\varepsilon\to0.
	\end{equation}
	Then we have the regularized system as follows:
	\begin{align}\label{eq4.5}
		\left\{
		\begin{aligned}
			&\partial_t f_\varepsilon+\hat{v}(w) \cdot \nabla_x f_\varepsilon+\nabla_w \cdot (F[f_\varepsilon,u_\varepsilon]f_\varepsilon)=0,\\
			&\partial_t u_\varepsilon+((m_\varepsilon \ast u_\varepsilon)\cdot\nabla)u_\varepsilon+\nabla p_\varepsilon-\mu \Delta u_\varepsilon=-\int_{\mathbb{R}^3}F_d[u_\varepsilon]f_\varepsilon \gamma_\varepsilon(w) \mathrm{d}w,\\
			&\nabla\cdot u_\varepsilon=0, (t,x,w)\in \mathbb{R}\times\mathbb{T}^3\times\mathbb{R}^3,
		\end{aligned}
		\right.
	\end{align}
	subject to regularized initial data:
	\begin{equation}
		f_\varepsilon(0,x,w)=f_\varepsilon^\mathrm{in}(x,w),\quad u_\varepsilon(0,x)=u_\varepsilon^\mathrm{in}(x),
	\end{equation}
	where $\ast$ denotes the convolution with respect to the spatial variable $x$, and
	\begin{equation}
		F[f_\varepsilon,u_\varepsilon](t,x,w)=F[f_\varepsilon](t,x,w)+m_\varepsilon*u_\varepsilon(t,x)-w,
	\end{equation}
	and $f_\varepsilon^\mathrm{in}$ and $u_\varepsilon^\mathrm{in}$ are $C^\infty$ approximations of $f^\mathrm{in}$ and $u^\mathrm{in}$ such that $\displaystyle{\int_{\mathbb{T}^3\times\mathbb{R}^3}f_\varepsilon^\mathrm{in}}$ d$z=1$ for any $\varepsilon>0$, $f_\varepsilon^\mathrm{in}$ converges to $f^\mathrm{in}$ strongly in $L^p(\mathbb{T}^3\times\mathbb{R}^3)$ for all $p\in[1,\infty)$, and weakly $\operatorname{in}L^\infty(\mathbb{T}^3\times\mathbb{R}^3)$, $M_2f_\varepsilon^\mathrm{in}$ is uniformly bounded with respect to $\varepsilon$ and converges strongly towards $M_2f^\mathrm{in}$ in $L^\infty(\mathbb{T}^3)$ and $(u_\varepsilon^\mathrm{in})$ strongly converges to $u^\mathrm{in}$ in $L^2(\mathbb{T}^3).$ Furthermore, we assume that $f_\varepsilon^\mathrm{in}$ has a compact support in $w$.
	
	Similarly to Definition \ref{df4.1}, we define the weak solution of the regularization system \eqref{eq4.5} as follows:
	\begin{df}\label{df4.2}
		For a given $T\in(0,\infty)$, the pair $(f_\varepsilon,u_\varepsilon)$ is a weak solution of \eqref{eq1} on the time interval $[0,T)$ if and only if the pair satisfies the following conditions:
		
		(1) $f_\varepsilon\in L^\infty(0,T;L^\infty(\mathbb{T}^3\times\mathbb{R}^3))$,
		
		(2) $u_\varepsilon\in \left(L^2(0,T;H^1(\mathbb{T}^3))\cap L^\infty(0,T;L^2(\mathbb{T}^3))\right)$,
		
		(3)  $\partial_tu_\varepsilon\in L^2(0,T;L^2(\mathbb{T}^3))$,
		
		(4) For all $\varphi\in\mathcal{C}^1([0,T]\times\mathbb{T}^3\times\mathbb{R}^3)$,  for a.e. $t\in[0,T]$,
		\begin{align}
			&\int_{\mathbb{T}^3\times\mathbb{R}^3}f_\varepsilon^\mathrm{in}(t)\varphi(t,\cdot,\cdot)\:\mathrm{d}x\mathrm{d}w-\int_0^t\int_{\mathbb{T}^3\times\mathbb{R}^3}f_\varepsilon\left(\partial_t\varphi+\hat{v}(w)\cdot\nabla\varphi+F[f_\varepsilon,u_\varepsilon]\cdot\nabla_w\varphi\right)\mathrm{d}x\mathrm{d}w\mathrm{d}s\nonumber\\
			&=\int_{\mathbb{T}^3\times\mathbb{R}^3}f_\varepsilon^\mathrm{in}\varphi(0,\cdot,\cdot)\:\mathrm{d}x\mathrm{d}w\nonumber,
		\end{align}
		
		(5) For all $\psi\in\mathcal{C}^1(\mathbb{T}^3\times[0,T])$ such that $\nabla\cdot\psi=0$,for a.e. $t \in [0,T]$,
		$$\begin{aligned}\int_0^t\int_{\mathbb{T}^3}\partial_tu_\varepsilon(t)\cdot\psi(t)\mathrm{d}x&+\int_0^t\int_{\mathbb{T}^3}\left((m_\varepsilon\ast u_\varepsilon\cdot\nabla)u\cdot\psi+\nabla u_\varepsilon:\nabla\psi\right)\mathrm{d}x\mathrm{d}s\\&=-\int_0^t\int_{\mathbb{T}^3\times\mathbb{R}^3}(u_\varepsilon-w)\cdot\psi f\gamma_\varepsilon(w)\:\mathrm{d}x\mathrm{d}w\mathrm{d}s.\end{aligned}$$
	\end{df}
	
	To solve the regularized system, we construct a sequence of approximate solutions $\{[f_\varepsilon^n,u_\varepsilon^n]\}_{n=1}^\infty$ for the regularized system inductively, as we discussed. Here, we omit $\varepsilon$- dependence of $f_\varepsilon^n$ and $u_\varepsilon^n$ for simplicity:
	\begin{equation}
		f^n:=f_\varepsilon^n,\quad u^n:=u_\varepsilon^n.\nonumber
	\end{equation}
	
	$\bullet$ (Initial step) We set $u^1:=u_\varepsilon^\mathrm{in}$ and solve the following Relativistic-Cucker-Smale part:
	$$\partial_tf^1+\hat{v}(w)\cdot\nabla f^1+\nabla_w\cdot(F[f^1,u^1]f^1)=0,\quad(x,w)\in\mathbb{T}^3\times\mathbb{R}^3,$$
	subject to the initial data:
	
	$$f^1(0,x,w):=f_\varepsilon^{\mathrm{in}}(x,w).$$
	
	$\bullet$ (Induction step) Suppose we have approximate solutions $\{[f^k,u^k]\}_{k=1}^n.$ Then with$[f^n,u^n]$,
	we solve the Navier-Stokes equation with drag force:
	\begin{align}\label{eq4.8}
		&\partial_tu^{n+1}+(m_\varepsilon*u^{n+1}\cdot\nabla)u^{n+1}+\nabla p^{n+1}-\Delta u^{n+1}=-\int_{\mathbb{R}^3}f^n(u^n-w)\gamma_\varepsilon(w)\mathrm{d}w,\nonumber\\
		&\nabla\cdot u^{n+1}=0,
	\end{align}
	subject to the initial data
	$$u^{n+1}(x,0):=u_\varepsilon^{\mathrm{in}}(x).$$
	With $u^{n+1}$, we again solve the following Relativistic-Cucker-Smale part:
	\begin{equation}\label{eq4.9}
		\partial f^{n+1}+\hat{v}(w)\cdot\nabla f^{n+1}+\nabla_{w}\cdot(F[f^{n+1},u^{n+1}]f^{n+1})=0,\quad(x,w)\in\mathbb{T}^3\times\mathbb{R}^3,
	\end{equation}
	subject to the initial data
	$$f^{n+1}(0,x,w):=f_\varepsilon^{\mathrm{in}}(x,w).$$
	
	Thus, we can construct $[f^{n+1},u^{n+1}]$ from $[f^n,u^n].$ 
	\subsection{The solvability of the regularized system}
	For this iteration to be well-defined, we need to check the solvability of \eqref{eq4.8} and \eqref{eq4.9}. Note that the existence of solution for the Navier-Stokes part \eqref{eq1} can be obtained by using the same argument as in \cite{ref29}. For the Relativistic-Cucker-Smale part \eqref{eq4.8}, we present its solvability at first. More precisely, for a given bulk velocity field $u\in L^\infty(0,T;L^2(\mathbb{T}^3))\cap L^2(0,T;H^1(\mathbb{T}^3))$, we show the global existence of $W^{1,\infty}$ solutions to \eqref{eq4.9} with regularized initial data:
	\begin{equation}\label{syt2}
		\begin{cases}
			\partial_tf+\hat{v}(w)\cdot\nabla f+\nabla_w\cdot(F[f,u]f)=0,\quad(x,w)\in\mathbb{T}^3\times\mathbb{R}^3,\\f(0,x,w):=f_\varepsilon^{\text{in}}(x,w).
		\end{cases}
	\end{equation}
	
	We use the method of characteristics to prove the existence of classical solution and $W^{1,\infty}$- estimates. For this, we first introduce forward characteristics $(x(t),w(t)):=\left(x(t;0,x,w,\theta),w(t;0,x,w,\theta),\theta(t;0,x,w,\theta)\right)$ to be the solution of the following ODEs:
	\begin{equation}\label{ode1}
		\begin{cases}
			\displaystyle{\frac{\mathrm{d}x(t)}{\mathrm{d}t}=\hat{v}(w(t))=:v(t),\quad t>0},\\
			\displaystyle{\frac{\mathrm{d}w(t)}{\mathrm{d}t}=L[f](t,x(t),w(t))+(m_\varepsilon\ast u)(t,x(t))-w(t)},\\
		\end{cases}
	\end{equation}
	subject to initial data:
	$$(x(0),w(0))=(x,w).$$
	For notational simplicity, we set $z(t):=(x(t),w(t)).$ Then, along the characteristics, we
	have
	$$\begin{aligned}\frac{\mathrm{d}}{\mathrm{d}t}f(z(t),t)&=\partial_t f(z(t),t)+\hat{v}(w(t))\cdot\nabla f(z(t),t)+\frac{\mathrm{d}w(t)}{\mathrm{d}t}\cdot\nabla_w f(z(t),t)\\&=-\nabla_w\cdot \left(L[f](t,x(t),w(t))+(m_\varepsilon\ast u)(t,x(t))-w(t)\right)\\&=\left(3+\int_{\mathbb{T}^3\times\mathbb{R}^3}\phi(|x(t)-x_*|)(\nabla_w\cdot v(t))f(t,z_*)\:\mathrm{d}x_*\mathrm{d}w_*\right)f(z(t),t).\end{aligned}$$
	We integrate the above characteristic equation to find
	\begin{equation}\label{eq4.12}
		f(t,z(t))=f_\varepsilon^{\text{in}}(z)\exp\left(3\left(t+\int_0^t\int_{\mathbb{T}^3\times\mathbb{R}^3}\phi(|x(t)-x_*|)(\nabla_w\cdot v(t))f(t,z_*)\:\mathrm{d}x_*\mathrm{d}w_*\right)\right).
	\end{equation}
	
	\begin{lem}\label{lem4.1}
		Let $f$ be a solution to \eqref{eq4.5} which decays to zero sufficiently fast in the w-variable. We set $\displaystyle{M_0(0)=1}$. Then, we have
		\begin{equation}
			M_0(t)=1,\quad \left|M_1(t)\right|\leq \sqrt{M_2(0)}+\|u\|_{L^\infty(0,T;L^2(\mathbb{T}^3))}, \quad\text{and}\quad M_2(t)\leq (\sqrt{M_2(0)}+\|u\|_{L^\infty(0,T;L^2(\mathbb{T}^3))})^2.\nonumber
		\end{equation}
	\end{lem}
	\begin{proof}
		(1) From the law of conservation of mass in Lemma \ref{lem2.1}, we know that $$M_0(t)=M_0(0)=1.$$
		(2) A straightforward computation yields
		\begin{align}\label{est1}
			\frac{\mathrm{d}}{\mathrm{d}t}M_2(t)&=-\int_{\mathbb{T}^3\times\mathbb{R}^3}|w|^2\nabla_w\cdot (F[f,u]f)\mathrm{d}x\mathrm{d}w\nonumber\\    &=2\int_{\mathbb{T}^3\times\mathbb{R}^3}w\cdot(L[f]f)\mathrm{d}x\mathrm{d}w+2\int_{\mathbb{T}^3\times\mathbb{R}^3}w\cdot(F_df)\mathrm{d}x\mathrm{d}w\nonumber\\
			&=-\int_{\mathbb{T}^{6}\times\mathbb{R}^{6}}\phi(|x-x_*|)(w-w_*)\cdot(\hat{v}(w)-\hat{v}(w_*)f(t,x_*,w_*)f(t,x,w)\mathrm{d}x\mathrm{d}w\mathrm{d}x_*\mathrm{d}w_*\nonumber\\
			&\quad+2\int_{\mathbb{T}^3\times\mathbb{R}^3}w\cdot(m_\varepsilon \ast u-w)f\mathrm{d}w\mathrm{d}x\nonumber\\
			&\leq -\frac{c^2+1}{c^2}\int_{\mathbb{T}^{6}\times\mathbb{R}^{6}}\phi(|x-x_*|)|\hat{v}(w)-\hat{v}(w_*)|^2f(t,x_*,w_*)f(t,x,w)\mathrm{d}x\mathrm{d}w\mathrm{d}x_*\mathrm{d}w_*\nonumber\\
			&\quad+2\int_{\mathbb{T}^3\times\mathbb{R}^3}w\cdot(m_\varepsilon \ast u-w)f\mathrm{d}w\mathrm{d}x\nonumber\\
			&\leq 2\int_{\mathbb{T}^3\times\mathbb{R}^3}w\cdot(m_\varepsilon \ast u-w)f\mathrm{d}w\mathrm{d}x=2\int_{\mathbb{T}^3\times\mathbb{R}^3}w\cdot (m_\varepsilon \ast u) f\mathrm{d}w\mathrm{d}x-2M_2(t)\nonumber\\
			&\leq 2\|u\|_{L^\infty(0,T;L^2(\mathbb{T}^3))}\int_{\mathbb{T}^3\times\mathbb{R}^3}|w| f\mathrm{d}w\mathrm{d}x-2M_2(t)\nonumber\\
			&\leq 2\|u\|_{L^\infty(0,T;L^2(\mathbb{T}^3))}\sqrt{M_2(t)}-2M_2(t)
		\end{align}
		We set $g(t)=\sqrt{M_2(t)}$,then by applying $Gr\ddot{o}nwall's$ inequality, the estimate \eqref{est1} yields
		\begin{equation}
			\frac{\mathrm{d}}{\mathrm{d}t}g(t)\leq \|u\|_{L^\infty(0,T;L^2(\mathbb{T}^3))}-g(t) \Rightarrow |g(t)|^2=M_2(t)\leq (\sqrt{M_2(0)}+\|u\|_{L^\infty(0,T;L^2(\mathbb{T}^3))})^2
		\end{equation}
		(3)By the $H\ddot{o}lder$ inequality, we have
		\begin{align}
			|M_1(t)|&\leq \int_{\mathbb{T}^3\times\mathbb{R}^3}|w| f\mathrm{d}w\mathrm{d}x\leq \left(\int_{\mathbb{T}^3\times\mathbb{R}^3}|w|^2 f\mathrm{d}w\mathrm{d}x\right)^\frac{1}{2}\left(\int_{\mathbb{T}^3\times\mathbb{R}^3} f\mathrm{d}w\mathrm{d}x\right)^\frac{1}{2}\nonumber\\
			&=\sqrt{M_2(t)}\leq \sqrt{M_2(0)}+\|u\|_{L^\infty(0,T;L^2(\mathbb{T}^3))}
		\end{align}
	\end{proof}
	\begin{lem}\label{lem4.2}
		Let $z(t):=(x(t),w(t))$ be a particle trajectory starting from
		$(x,w)\in$supp$f^\mathrm{in}$. Then we have
		$$|w(t)|\leq C_W,\quad t\leq T,$$
		where $C_W$ is a constant which depends only on $\|u\|_{L^\infty(0,T;L^2(\mathbb{T}^3))},\phi,M_2(0),\Omega_w(f,0)$ and $T$.
	\end{lem}
	\begin{proof}
		For simplicity, we note that $$F(v(t))=:F(t)\geq 1,\quad\text{and}\quad F(v_*)=:F_*\geq 1.$$
		From the characteristic equation \eqref{ode1}, we use $F(\cdot)\geq1$ to get
		\begin{align}\label{eq4.16}
			\frac{\mathrm{d}}{\mathrm{d}t}|w^i(s)|&=\left((L)^i[f](z(s),s)+(m_\varepsilon \ast u^i)(z(s),s)-w^i(s)\right)\mathrm{sgn}(w^i(s))\nonumber\\
			&\leq\left|\int_{\mathbb{T}^3\times\mathbb{R}^3}\phi(x(s)-x_*)\left(\frac{w_*}{F_*}\right)^if(z_*,s)\:\mathrm{d}x_*\mathrm{d}w_*\right|\nonumber\\
			&-\left(\frac{1}{F(s)}\int_{\mathbb{T}^3\times\mathbb{R}^3}\phi(x(s)-x_*)f(z_*,s)\:\mathrm{d}x_*\mathrm{d}w_*+1\right)|w^i(s)|+|(m_\varepsilon\ast u^i)(x(s),s)|\nonumber\\
			&\leq \frac{\phi(0)}{F_*}|M_1(s)|-|w^i(s)|+C(\varepsilon)\|u\|_{L^\infty(0,T,L^2(\mathbb{T}^3))}\nonumber\\
			&\leq\phi(0)\left(\left(\sqrt{M_2(0)}+\|u\|_{L^\infty(0,T;L^2(\mathbb{T}^2))}\right)-|w^i(s)|+C(\varepsilon)\|u\|_{L^\infty(0,T;L^2(\mathbb{T}^2))}\right.
		\end{align}
		for $i=1,2,3$.
		
		We apply $Gr\ddot{o}nwall's$ Lemma to the above relation to obtain
		$$|w^i(s)|\leq|w^i(0)|+\phi(0)\left(\sqrt{M_2(0)}+\|u\|_{L^\infty(0,T;L^2(\mathbb{T}^3))}\right)+C(\varepsilon)\|u\|_{L^\infty(0,T;L^2(\mathbb{T}^3))},$$which yields the desired estimate.
	\end{proof} 
	\begin{rem}\label{rem4.1}
		1. We define $\lambda_w(s)=\sup\limits_{w\in\Omega_w(f,s)}|w|$ and $\lambda_v(s)=\sup\limits_{w\in\Omega_w(f,s)}|\hat{v}(w)|$. Then, Lemma \ref{lem4.2} implies that $\lambda_w(s)\leq C_W$ and $\lambda_v(s)\leq g^{-1}(C_W) < c$ (light speed constant),
		where $C_W$ depends only on $\lambda_W(0)$, $\|u\|_{L^\infty(0,T;L^2(\mathbb{T}^3))}$, $\phi(0),M_0(0)$ and $M_2(0)$.
		
		2. Let $[x(s),w(s)]$ be the particle trajectories defined by \eqref{ode1}. In fact, the constant $C$ in
		Lemma \ref{lem4.2} does not depend on $n$ due to the uniform estimate \eqref{eq4.16}.
	\end{rem}
	We next present $L^\infty$-estimates for the derivatives of the force field $F$:
	\begin{equation}\label{eq13}
		F[f,u]=L[f](t,x,w)+(m_\varepsilon\ast u)(t,x)-w
	\end{equation}
	\begin{lem}\label{lem4.3}
		The regularized force term $F[f,u]$ satisfies
		
		$$\begin{aligned}&(1)\quad\|\nabla_{w}\cdot F\|_{L^{\infty}}\leq 3(\phi(0)+1),\\
			&(2)\quad\|\partial_{w_i} F\|_{L^\infty}\leq \sqrt{3}\phi(0)+1,\\
			&(3)\quad\|\nabla_{x}\cdot F\|_{L^{\infty}}\leq 6\|\phi^{\prime}\|_{L^\infty}\lambda_v(t)+3\|u\|_{L^\infty(0,T;L^2(\mathbb{T}^3))}\|\nabla m_\varepsilon\|_{L^1},\\
			&(4)\quad\|\nabla_{w}\cdot\partial_{x_i}F\|_{L^{\infty}}\leq 3\|\phi^{\prime}\|_{L^{\infty}}.\end{aligned}$$
	\end{lem}
	\begin{proof}
		(1) We recall the fact $\nabla_w \cdot \hat{v}(w)=3-|O(c^{-2})|$. From Lemma \ref{lem4.2} and Remark \ref{rem4.1} we know that $\hat{v}(w)$ is bounded in $[0,T]$. Therefore, by direct calculation, we have
		$$\nabla_w\cdot F=\int_{\mathbb{T}^3\times\mathbb{R}^3}\phi(|x-x_*|)(\nabla_w \cdot \hat{v}(w))f\mathrm{d}x_*\mathrm{d}w_*-3.$$
		This yields
		$$\|\nabla_w\cdot F\|_{L^\infty}\leq (3-|O(c^{-2})|)\phi(0)+3\leq 3(\phi(0)+1).$$
		(2)We observe that
		$$|\partial_{{w_{i}}}\hat{v}(w)|\leq\|\nabla_{w}\hat{v}\|_{1}\leq\sqrt{3}\|\nabla_{w}\hat{v}\|_{op}=\sqrt{3}\lambda_{1}^{-1}\leq\sqrt{3}$$
		where $\displaystyle{\|A\|_{1}:=\max\limits_{1\leq j\leq 3}\sum_{i=1}^{3}|A_{ij}}|$ is an absolute column sum of the matrix.\\
		This yields
		$$\|\partial_{w_i} F\|_{L^\infty}\leq \sqrt{3}\phi(0)+1.$$
		(3) Again it follows from \eqref{eq13} that for $(x,w)\in\mathbb{T}^3\times\Omega_w(f,t)$,
		$$\nabla_x\cdot F=\sum_{i=1}^3\int_{\mathbb{T}^3\times\mathbb{R}^3}\phi^{\prime}(|x-x_*|)\frac{x^i-y^i}{|x-x_*|}([\hat{v}(w_*)]^i-[\hat{v}(w)]^i)f\mathrm{d}x_*\mathrm{d}w_*+\sum_{i=1}^3(\nabla m_\varepsilon)*u^i.$$
		This implies
		$$\|\nabla_x\cdot F\|_{L^\infty}\leq 6\|\phi^{\prime}\|_{L^\infty}\lambda_v(t)+3\|u\|_{L^\infty(0,T;L^2(\mathbb{T}^3))}\|\nabla m_\varepsilon\|_{L^1}.$$
		(4) Since
		$$\partial_{x_i}F=\int_{\mathbb{T}^3\times\mathbb{R}^3}\phi^{\prime}(|x-x_*|)\frac{x^i-x_*^i}{|x-x_*|}(\hat{v}(w_*)-\hat{v}(w))f\mathrm{d}x_*\mathrm{d}w_*+\nabla m_\varepsilon*u^i,$$
		we have
		$$\nabla_w\cdot\partial_{x_i}F=\int_{\mathbb{T}^3\times\mathbb{R}^3}\phi^{\prime}(|x-x_*|)\frac{x^i-x_*^i}{|x-x_*|}(-3+|O(c^{-2})|)f\mathrm{d}x_*\mathrm{d}w_*.$$
		This yields
		$$\|\nabla_w\cdot\partial_{x_i}F\|_{L^\infty}\leq 3\|\phi'\|_{L^\infty}.$$
	\end{proof}
	\begin{lem}
		Suppose that the initial datum $f^\mathrm{in}\in C^1(\mathbb{T}^{3}\times\mathbb{R}^3)$ has a compact support. Then, for any fixed $T\in(0,\infty),there$ exists a unique classical solution $f\in\mathcal{C}([0,T)\times\mathbb{T}^{3}\times\mathbb{R}^3)$ to \eqref{syt2}.
	\end{lem}
	\begin{proof}
		Since the existence part is based on the method of successive approximations, we only provide $W^{1,\infty}$-norm estimate in
		the sequel. A higher order $W^{k,\infty}$-norm estimate with $k\geq2$ can be done similarly. To this end, we define the transport-like operator $\mathcal{T}$ as
		$$\mathcal{T}:=\partial_t+\hat{v}(w)\cdot\nabla_x+L[f]\cdot\nabla_w.$$
		
		Then, we have
		\begin{align}
			&\mathcal{T}(f)=-(\nabla_w\cdot F)f,\nonumber\\
			&\mathcal{T}(\partial_{x_i}f)=-(\partial_{x_i}F)\cdot\nabla_wf-\partial_{x_i}(\nabla_w\cdot F)f-(\nabla_w\cdot F)(\partial_{x_i}f),\nonumber\\
			&\mathcal{T}(\partial_{w_i}f)=-(\partial_{w_i}\hat{v}(w))\cdot\nabla_xf-(\partial_{w_i}F)\cdot\nabla_wf-\partial_{w_i}(\nabla_w\cdot F)f-(\nabla_w\cdot F)(\partial_{w_i}f).\nonumber
		\end{align}
		From Lemma \ref{lem4.3}, it is easily to know that
		\begin{align}
			&|\mathcal{T}(f)|\leq C_1(\phi)\|f\|_{L^\infty} \\
			&|\mathcal{T}(\partial_{x_i}f)|\leq C_2(\phi,C_W,u,m_\varepsilon)(\|\nabla_w f\|_{L^\infty}+\|f\|_{L^\infty}+\|\nabla_x f\|_{L^\infty})
		\end{align}
		Moreover, observe that
		$$\partial_{w_i}\bigg(\nabla_w\cdot\hat{v}(w)\bigg)=-\frac{(d-1)\partial_{w_i}\lambda_1}{\lambda_1^2}-\frac{\partial_{w_i}\lambda_2}{\lambda_2^2}.$$
		However, since
		\begin{equation}\label{eq4.19}
			\left|\frac{\partial\Gamma}{\partial w_i}\right|=\frac{\Gamma|\upsilon_{i}|}{c^{2}\Gamma+2\Gamma^{2}-1}\leq\frac{\Gamma|\upsilon|}{c^{2}\Gamma+2\Gamma^{2}-1}=\frac{c\sqrt{\Gamma^{2}-1}}{c^{2}\Gamma+2\Gamma^{2}-1}\leq\frac{1}{c},
		\end{equation}
		the terms in \eqref{eq4.19} can be bounded as
		$$\left|\frac{\partial\lambda_1}{\partial w_i}\right|=\left|\frac{\partial F}{\partial w_i}\right|=\left(1+\frac{2\Gamma}{c^2}\right)\left|\frac{\partial\Gamma}{\partial w_i}\right|\leq\frac1c{\left(1+\frac{2\Gamma}{c^2}\right)}$$
		and
		$$\left|\frac{\partial\lambda_2}{\partial w_i}\right|\leq\left|\frac{\partial F}{\partial w_{i}}\right|+\left(3\Gamma^{2}-1+\frac{8\Gamma^{3}}{c^{2}}-\frac{4\Gamma}{c^{2}}\right)\left|\frac{\partial\Gamma}{\partial w_{i}}\right|\leq\frac{1}{c}\left(3\Gamma^{2}+\frac{8\Gamma^{3}}{c^{2}}-\frac{2\Gamma}{c^{2}}\right).$$
		Therefore, $\partial_{w_i}\lambda_k$ is bounded, then $\displaystyle{\partial_{w_i}\bigg(\nabla_w\cdot\hat{v}(w)\bigg)}$ is also bounded. Finally, $\partial_{w_i}(\nabla_w\cdot F)f$ is bounded and we have
		$$|\mathcal{T}(\partial_{w_{i}}f)|\leq C_3(\phi,C_W,u,m_\varepsilon)(\|f\|_{L^{\infty}}+\|\nabla_{x}f\|_{L^{\infty}}+\|\nabla_{w}f\|_{L^{\infty}}).$$
		Next, we introduce a functional $\mathcal{F}$,
		
		$$\mathcal{F}(t):=\|f(t,\cdot,\cdot)\|_{L^\infty}+\|\nabla_x f(t,\cdot,\cdot)\|_{L^\infty}+\|\nabla_w f(t,\cdot,\cdot)\|_{L^\infty}.$$
		
		Then, we combine all the estimates to conclude that $\mathcal{F}$ satisfies the Grönwall inequality,
		
		$$\frac{d\mathcal{F}}{\mathrm{d}t}\leq C(\phi,C_W,u,m_\varepsilon)\mathcal{F}(t),\quad t>0.$$
		
		This implies
		\begin{equation}\label{eq4.21}
			\mathcal{F}(t)\leq e^{CT}\mathcal{F}(0)\quad\mathrm{for}\quad0\leq t\leq T.
		\end{equation}
		This directly leads to the boundedness of the $W^{1,\infty}$-norm of $f$ in any finite time interval.
		
		For the uniqueness, since the operator $\mathcal{T}$ is linear, we can use a similar argument to derive a kind of Grönwall-type inequality \eqref{eq4.21} for the difference of two solutions with the same initial datum. Therefore, we can obtain the uniqueness.
	\end{proof}
	
	\subsection{Convergence of iteration}
	
	We next provide the uniform bound estimates of the approximate solutions to system \eqref{eq4.8}
	and\eqref{eq4.9}. For this, we consider the following forward characteristics:
	$$(x^{n+1}(s),w^{n+1}(s)):=(x^{n+1}(s;0,x,w),w^{n+1}(s;0,x,w)),$$
	which is a solution to the following integro-differential system:
	\begin{align}\label{eq4.22}
		&\frac{\mathrm{d}}{\mathrm{d}s}x^{n+1}(s)=\hat{v}(w^{n+1}(s)):=v^{n+1}(s),\quad s>0,\nonumber\\
		&\frac{\mathrm{d}}{\mathrm{d}s}w^{n+1}(s)=L[f^n](s,x^{n+1}(s),w^{n+1}(s))+(m_\varepsilon*u^n)(s,x^{n+1}(s))-w^{n+1}(s),\nonumber\\
	\end{align}
	with the initial data:
	$$\left.(x^{n+1},w^{n+1})\right|_{t=0}=(x,w)\quad\text{for all}\quad n\geq0.$$
	For the sake of notational simplicity, we set $z^{n+1}(s):=(x^{n+1}(s),w^{n+1}(s)).$
	
	In order to eliminate the dependence on n, we need to have the following uniform bound of the sequences $[f^n,u^n]$.
	
	\begin{lem}\label{lem4.5}
		For each $n\in \mathbb{N}$ , there exists a unique solution $[ f^n, u^n]$ of \eqref{eq4.8} and \eqref{eq4.9} in $\mathcal{C}^1([0,T];\mathcal{C}^1(\mathbb{T}^3\times\mathbb{R}^3))\times[(H^1(0,T;L^2(\mathbb{T}^3))\cap L^2(0,T;H^1(\mathbb{T}^3))]$ satisfying the following estimates:
		$$\begin{aligned}&(1)\quad\|f^n\|_{L^\infty(\mathbb{T}^3\times\mathbb{R}^3\times(0,T))}\leq C(T)\|f_{\mathrm{in},\varepsilon}\|_{L^\infty(\mathbb{T}^3\times\mathbb{R}^3)},\\&(2)\quad\|u^n\|_{L^\infty(0,T;L^2(\mathbb{T}^3))\cap L^2(0,T;H^1(\mathbb{T}^3))}\leq C(\varepsilon),\\&(3)\quad\|\partial_tu^{n+1}\|_{L^2(0,T;L^2(\mathbb{T}^3))}\leq C(\varepsilon),\end{aligned}$$
		where $C(\varepsilon)$ depends on $\varepsilon$ ut not on $n$, and the sequence $[f^n,u^n]$ strongly converges in
		$L^\infty ( \mathbb{T}^3\times \mathbb{R}^3\times(0,T)) \times L^\infty (0,T; L^2(\mathbb{T}^3)) $ towards weak formulation of \eqref{eq4.8} - \eqref{eq4.9}. 
	\end{lem}
	\begin{proof}
		Assume that $[f^n,u^n]$ is given in $L^\infty((0,T)\times\mathbb{T}^3\times\mathbb{R}^3)\times L^\infty(0,T;L^2(\mathbb{T}^3)).$ Since we cut off the high particle velocities, the drag force $\displaystyle{\int_{\mathbb{R}^3}f^n\left(w-u^n\right)\gamma(w)\mathrm{d}w}$ at least belongs to $L^2(0,T;L^2(\mathbb{T}^3)).$ Applying standard results about the Navier-Stokes system (with a mollified convection term), one can prove the existence and uniqueness of $u^{n+1}\in L^\infty(0,T;L^2(\mathbb{T}^3))\cap L^2(0,T;H^1(\mathbb{T}^3))$ solving \eqref{eq4.5}$_2$.
		
		Note that the uniqueness relies on the fact that the convection velocity has been regularized. Moreover, since the convection velocity has been regularized and the initial velocity is smooth, one can check that $\partial_tu^n\in L^2(0,T;L^2(\mathbb{T}^3)).$ By a simple induction argument, these properties hold for all $n\in\mathbb{N}.$
		
		We have from \eqref{eq4.12} that for $\forall (t,x,w) \in (0,T)\times\mathbb{T}^3\times\mathbb{R}^3$,
		$$f^{n}(t,x,w)\leq\mathrm{e}^{3T(\|\phi\|_{\mathcal{C}^1}+1)}\|f_{\mathrm{in},\varepsilon}\|_{L^\infty(\mathbb{T}^3\times\mathbb{R}^3)},$$
		which implies (1). Then, we can derive the other uniform estimates in a similar way as in Subsection 3.1.2 in \cite{ref29}.
	\end{proof}
	\begin{rem}\label{rem4.2}
		We define $\lambda_w^n(s)=\sup\limits_{w\in\Omega_w(f^n,s)}|w|$ and $\lambda_v^n(s)=\sup\limits_{w\in\Omega_w(f^n,s)}|\hat{v}(w)|$. Then, Lemma \ref{lem4.2} and Lemma \ref{lem4.5} implies that $\lambda_w^n(s)\leq C_W$ and $\lambda_v^n(s)\leq g^{-1}(C_W) < c$ (light speed constant). Here, the meaning of $C_W$ is the same as that of Lemma \ref{lem4.2}, and it is independent of $t$ and $n$.
	\end{rem}
	Next, we show that the approximation sequence $[f^n,u^n]$  is a Cauchy sequence in $L^\infty([0,T]\times\mathbb{T}^3\times\mathbb{R}^3))\times L^\infty(0,T;L^2(\mathbb{T}^3))$.  For this, we set $\omega^{n+1}:=u^{n+1}-u^n$.
	\begin{lem}\label{lem4.6}
		For a given positive number $T$, let $[f^n,u^n]$ be the $n$-th iterate which is a solution to system \eqref{eq4.8} and \eqref{eq4.9} in the time interval $[0,T)$, and $z^n(t)$ is the forward trajectory of particles associated with \eqref{eq4.9} and issued from $(x,w)$ at time 0. Then for $t \in [0,T)$, we have
		$$\begin{aligned}&(1)\quad\|f^n(t)-f^{n-1}(t)\|_{L^\infty}+\|z^n(t)-z^{n-1}(t)\|_{L^\infty}\leq C(\varepsilon,T)\int_0^t\|\omega^n(s)\|_{L^2}\mathrm{d}s,\\&(2)\quad\|\omega^{n+1}(t)\|_{L^2}^2\leq C(\varepsilon,T)\left(\int_0^t\|\omega^n(s)\|_{L^2}^2\mathrm{d}s+\int_0^t\|\omega^{n+1}(s)\|_{L^2}^2\mathrm{d}s\right).\end{aligned}$$
	\end{lem}
	\begin{proof}
		For the estimate, we take two straightforward steps.
		
		Step $A{:}$ We first derive
		\begin{align}\label{eq4.23}
			&\|f^n(t)-f^{n-1}(t)\|_{L^\infty}\leq C(T)\|z^n(t)-z^{n-1}(t)\|_{L^\infty}\nonumber\\
			&\quad\quad+C(c,\phi, C_W,\varepsilon,T)\biggl(\int_0^t\|z^n(s)-z^{n-1}(s)\|_{L^\infty}\:\mathrm{d}s+\int_0^t\|f^n(s)-f^{n-1}(s)\|_{L^\infty}\:\mathrm{d}s\biggr).\end{align}
		
		We integrate equation \eqref{eq4.9} along the particle trajectories $z^n(s)$ and $z^{n-1}$ from $s=0$ to $s=t$ to find
		\begin{align}
			&f^n(z^n(t),t)-f^{n-1}(z^{n-1}(t),t)\nonumber\\
			=&\mathrm{e}^{3t}f_{\mathrm{in},\varepsilon}(x,w)\times\bigg(\exp\left\{\int_0^t\int_{\mathbb{T}^3\times\mathbb{R}^3}\phi(|x^n(s)-x_*|)\nabla_w\cdot v^{n}(s)f^n\mathrm{d}x_*\mathrm{d}w_*\mathrm{d}s\right\}\nonumber\\
			&\quad\quad-\exp\left\{\int_0^t\int_{\mathbb{T}^3\times\mathbb{R}^3}\phi(|x^{n-1}(s)-x_*|)f^{n-1}\nabla_w\cdot v^{n-1}(s)\mathrm{d}x_*\mathrm{d}w_*\mathrm{d}s\right\}\bigg).
		\end{align}
		
		This again can be rewritten as
		\begin{align}\label{eq4.24}
			&f^n(z^n(t),t)-f^{n-1}(z^n(t),t)\nonumber\\
			&=f^{n-1}(z^{n-1}(t),t)-f^{n-1}(z^n(t),t)\nonumber\\
			&\quad\quad+\mathrm{e}^{3t}f_{\mathrm{in},\varepsilon}(x,w)\times\bigg(\exp\left\{\int_0^t\int_{\mathbb{T}^3\times\mathbb{R}^3}\phi(|x^n(s)-x_*|)\nabla_w\cdot v^{n}(s)f^n\mathrm{d}x_*\mathrm{d}w_*\mathrm{d}s\right\}\nonumber\\
			&\quad\quad-\exp\left\{\int_0^t\int_{\mathbb{T}^3\times\mathbb{R}^3}\phi(|x^{n-1}(s)-x_*|)f^{n-1}\nabla_w\cdot v^{n-1}(s)\mathrm{d}x_*\mathrm{d}w_*\mathrm{d}s\right\}\bigg)\nonumber\\
			&=:\mathcal{K}_1(t)+\mathcal{K}_2(t).
		\end{align}
		
		$\bullet$(The estimate of $\mathcal{K}_1)$: we use \eqref{eq4.21} to obtain
		\begin{equation}\label{eq4.25}
			\mathcal{K}_1(t)\leq C(T)\|z^n(t)-z^{n-1}(t)\|_{L^\infty}.
		\end{equation}
		
		$\bullet$(The estimate of $\mathcal{K}_2)$: we use Remark \ref{rem2.1}, Remark \ref{rem4.2} and Lemma \ref{lem4.2} to obtain
		\begin{align}\label{eq4.26}
			\mathcal{K}_{2}(t)&\leq C(\varepsilon,T)\|f_{\mathrm{in},\varepsilon}\|_{L^{\infty}}\bigg(\int_0^t\int_{\mathbb{T}^3\times\mathbb{R}^3}\phi(|x^n(s)-x_*|)\nabla_w\cdot v^{n}(s)f^n\mathrm{d}x_*\mathrm{d}w_*\mathrm{d}s\nonumber\\
			&\quad\quad-\int_0^t\int_{\mathbb{T}^3\times\mathbb{R}^3}\phi(|x^{n-1}(s)-x_*|)f^{n-1}\nabla_w\cdot v^{n-1}(s)\mathrm{d}x_*\mathrm{d}w_*\mathrm{d}s\nonumber\\
			&\leq C(\varepsilon,T)\|f_{\mathrm{in},\varepsilon}\|_{L^{\infty}}\bigg(\int_0^t\int_{\mathbb{T}^3\times\mathbb{R}^3}\phi(|x^n(s)-x_*|)\nabla_w\cdot v^{n}(s)(f^n-f^{n-1})\mathrm{d}x_*\mathrm{d}w_*\mathrm{d}s\nonumber\\
			&\quad\quad +\int_0^t\int_{\mathbb{T}^3\times\mathbb{R}^3}\left(\phi(|x^n(s)-x_*|)\nabla_w\cdot v^{n}(s)-\phi(|x^{n-1}(s)-x_*|)\nabla_w\cdot v^{n-1}(s)\right)f^{n-1}\mathrm{d}x_*\mathrm{d}w_*\mathrm{d}s\bigg)\nonumber\\
			&\leq C(\varepsilon,T)\int_{0}^{t}\|f^{n}(s)-f^{n-1}(s)\|_{L^{\infty}}\mathrm{d}s\nonumber\\
			&\quad\quad+C(T)\int_{0}^{t}\|\phi(|x^n(s)-x_*|)\nabla_w\cdot v^{n}(s)-\phi(|x^{n-1}(s)-x_*|)\nabla_w\cdot v^{n-1}(s)\|_{L^{\infty}}\mathrm{d}s\nonumber\\
			&\leq C(\varepsilon,T)\int_{0}^{t}\|f^{n}(s)-f^{n-1}(s)\|_{L^{\infty}}\mathrm{d}s\nonumber\\
			&\quad\quad+C(T)\int_{0}^{t}\|\phi(|x^n(s)-x_*|)\left(\nabla_w\cdot v^{n}(s)-\nabla_w\cdot v^{n-1}(s)\right)\|_{L^{\infty}}\mathrm{d}s\nonumber\\
			&\quad\quad+C(T)\int_{0}^{t}\|\left(\phi(|x^n(s)-x_*|)-\phi(|x^{n-1}(s)-x_*|)\right)\nabla_w\cdot v^{n-1}(s)\|_{L^{\infty}}\mathrm{d}s\nonumber\\
			&\leq C(\varepsilon,T)\int_{0}^{t}\|f^{n}(s)-f^{n-1}(s)\|_{L^{\infty}}\mathrm{d}s\nonumber\\
			&\quad\quad+C(\phi, C_W,T)\int_{0}^{t}\|v^{n}(s)-v^{n-1}(s)\|_{L^{\infty}}\mathrm{d}s+C(\phi,T)\int_{0}^{t}\|x^n(s)-x^{n-1}(s)\|_{L^{\infty}}\mathrm{d}s\nonumber\\
			&\leq C(\varepsilon,T)\int_{0}^{t}\|f^{n}(s)-f^{n-1}(s)\|_{L^{\infty}}\mathrm{d}s\nonumber\\
			&\quad\quad+C(c,\phi, C_W,T)\int_{0}^{t}\|w^{n}(s)-w^{n-1}(s)\|_{L^{\infty}}\mathrm{d}s+C(\phi,T)\int_{0}^{t}\|x^n(s)-x^{n-1}(s)\|_{L^{\infty}}\mathrm{d}s\nonumber\\
			&\leq  C(\varepsilon,T)\int_{0}^{t}\|f^{n}(s)-f^{n-1}(s)\|_{L^{\infty}}\mathrm{d}s+C(c,\phi, C_W,T)\int_{0}^{t}\|z^{n}(s)-z^{n-1}(s)\|_{L^{\infty}}\mathrm{d}s
		\end{align}
		where we used the uniform boundednessof $w$-support of $f$ (see Remark \ref{rem4.1}). In \eqref{eq4.24}, we combine \eqref{eq4.25} and \eqref{eq4.26} to obtain the desired estimate \eqref{eq4.23}.
		
		Step $B{:}$ We next derive
		\begin{align}\label{eq4.28}
			\|z^n(t)-z^{n-1}(t)\|_{L^\infty}\leq C(\varepsilon,T)\bigg(\int_0^t\|\omega^n(s)\|_{L^2}\mathrm{d}s&+\int_0^t\|z^n(s)-z^{n-1}(s)\|_{L^\infty}\mathrm{d}s\nonumber\\
			&+\int_0^t\|f^n(s)-f^{n-1}(s)\|_{L^\infty}\mathrm{d}s\bigg).
		\end{align}
		
		It follows from \eqref{eq4.22} that we have the following estimate:
		\begin{align}|z^n(t)-z^{n-1}(t)|&\leq\int_0^t|v^n(s)-v^{n-1}(s)|\mathrm{d}s+\int_0^t|w^n(s)-w^{n-1}(s)|\mathrm{d}s\nonumber\\
			&\quad\quad+\int_0^t|(m_\varepsilon*u^n)(x^n(s),s)-(m_\varepsilon*u^{n-1})(x^{n-1}(s),s)|{d}s\nonumber\\
			&\quad\quad+\bigg|\int_0^t\int_{\mathbb{T}^3\times\mathbb{R}^3}\phi(|x^n(s)-x_*|)(\hat{v}(w_*)-v^n(s))f^n{d}x_*{d}w_*{d}s\nonumber\\
			&\quad\quad\quad-\int_0^t\int_{\mathbb{T}^3\times\mathbb{R}^3}\phi\left(|x^{n-1}(s)-x_*|\right)(\hat{v}(w_*)-v^{n-1}(s))f^{n-1}{d}x_*{d}w_*\mathrm{d}s\bigg|\nonumber\\
			&=:\mathcal{K}_3(t)+\mathcal{K}_4(t)+\mathcal{K}_5(t)+\mathcal{K}_6(t).
		\end{align}
		
		$\bullet$ (Estimates of $\mathcal{K}_i,i=3,4$): It can be easily estimated as
		$$\mathcal{K}_3(t)+\mathcal{K}_4(t)\leq C(c)\int_0^t\|z^n(s)-z^{n-1}(s)\|_{L^\infty}\mathrm{d}s.$$
		
		$\bullet$ (Estimates of $\mathcal{K}_5$): It can be easily estimated as
		\begin{align}
			\mathcal{K}_5(t)& \leq\int_0^t\left|(m_\varepsilon*u^n)(x^n(s),s)-(m_\varepsilon*u^n)(x^{n-1}(s),s)+m_\varepsilon*(u^n-u^{n-1})(x^{n-1}(s),s)\right|\mathrm{d}s \nonumber\\
			& \leq C(\varepsilon)\left(\int_0^t\|z^n(s)-z^{n-1}(s)\|_{L^\infty}\mathrm{d}s+\int_0^t\|\omega^n(s)\|_{L^2}\mathrm{d}s\right).
		\end{align}
		
		$\bullet$ (Estimates of $\mathcal{K}_6$): We split it into three terms:
		\begin{align}
			\mathcal{K}_{6}(t) & \leq\int_0^t\int_{\mathbb{T}^3\times\mathbb{R}^3}[\phi(|x^n(s)-x_*|)-\phi(|x^{n-1}(s)-x_*|)](\hat{v}(w_*)-v^n)f^n\mathrm{d}x_*\mathrm{d}w_*\mathrm{d}s \nonumber\\
			&\quad\quad +\int_0^t\int_{\mathbb{T}^3\times\mathbb{R}^3}\phi(|x^{n-1}(s)-x_*|)[v^{n-1}(s)-v^n]f^n\mathrm{d}x_*\mathrm{d}w_*\mathrm{d}s \nonumber\\
			&\quad\quad +\int_0^t\int_{\mathbb{T}^3\times\mathbb{R}^3}\phi(|x^{n-1}(s)-x_*|)(\hat{v}(w_*)-v^{n-1}(s))[f^n(s)-f^{n-1}(s)]\mathrm{d}x_*\mathrm{d}w_*\mathrm{d}s \nonumber\\
			& =:\mathcal{K}_{61}(t)+\mathcal{K}_{62}(t)+\mathcal{K}_{63}(t).
		\end{align}
		We now estimate the terms $\mathcal{K}_{6i},i=1,2,3$ as follows:
		\begin{align}
			\mathcal{K}_{61}(t)&\leq\int_0^t\left(\|z^n(s)-z^{n-1}(s)\|_{L^\infty}\int_{\mathbb{T}^3\times\mathbb{R}^3}|v_*|f^n\mathrm{d}x_*dv_*\right)\mathrm{d}s+C(\varepsilon)\int_0^t|v^n(s)|\left\|z^n(s)-z^{n-1}(s)\right\|_{L^\infty}\mathrm{d}s\nonumber\\
			&\leq\int_0^t\left(\|z^n(s)-z^{n-1}(s)\|_{L^\infty}\int_{\mathbb{T}^3\times\mathbb{R}^3}|w_*|f^n\mathrm{d}x_*dv_*\right)\mathrm{d}s+C(\varepsilon)\int_0^t|v^n(s)|\left\|z^n(s)-z^{n-1}(s)\right\|_{L^\infty}\mathrm{d}s\nonumber\\
			&\leq C(\varepsilon)\int_0^t\|z^n(s)-z^{n-1}(s)\|_{L^\infty}\mathrm{d}s,
		\end{align}
		and
		\begin{equation}
			\mathcal{K}_{62}(t)\leq C(\varepsilon)\int_0^t\|z^n(s)-z^{n-1}(s)\|_{L^\infty}\mathrm{d}s,
		\end{equation}
		\begin{equation}
			\mathcal{K}_{63}(t)\leq 
			C(\varepsilon)\int_0^t\|f^n(s)-f^{n-1}(s)\|_{L^\infty}\mathrm{d}s.
		\end{equation}
		where we used the uniform boundedness of $v^n$.
		
		We combine the above estimates to obtain the desired result \eqref{eq4.28}. We now consider the sum $\displaystyle{\eqref{eq4.23}+\frac1\eta\eqref{eq4.28}}$ with $\displaystyle{0<\eta\ll1}$ to obtain
		\begin{align}
			&\|f^{n}(t)-f^{n-1}(t)\|_{L^{\infty}}+\left(\frac{1}{\eta}-C(T)\right)\|z^{n}(t)-z^{n-1}(t)\|_{L^{\infty}}\nonumber\\
			&\leq\frac{C(\varepsilon,T)}{\eta}\biggl[\int_{0}^{t}\|\omega^{n}(s)\|_{L^{2}}\:\mathrm{d}s+\int_{0}^{t}\|z^{n}(s)-z^{n-1}(s)\|_{L^{\infty}}\:\mathrm{d}s+\int_{0}^{t}\|f^{n}(s)-f^{n-1}(s)\|_{L^{\infty}}\:\mathrm{d}s\biggr].\end{align}
		
		We choose $\eta\ll1$ sufficiently small such that
		
		$$\frac1\eta-C(T)\geq1.$$
		
		This together with setting $\Delta^n(t):=\|f^n(t)-f^{n-1}(t)\|_{L^\infty}+\|z^n(t)-z^{n-1}(t)\|_{L^\infty}$ gives
		
		$$\Delta^n(t)\leq\frac{C(\varepsilon,T)}{\eta}\bigg[\int_0^t\|\omega^n(s)\|_{L^2}\:\mathrm{d}s+\int_0^t\Delta^n(s)\:\mathrm{d}s\bigg],$$
		
		Then the standard $Gr\ddot{o}nwall's$ Lemma yields
		\begin{equation}\label{eq4.36}
			\Delta^n(t)\leq C(\varepsilon,T)\int_0^t\|\omega^n(s)\|_{L^2}\:\mathrm{d}s.
		\end{equation}
		
		For the estimate of $\omega^n$, we use the estimates of \cite{ref29}. In fact, we have
		
		$$\begin{aligned}\frac{1}{2}\frac{\mathrm{d}}{\mathrm{d}t}\|\omega^{n+1}(t)\|_{L^{2}}^{2}&+\frac{1}{2}\|\nabla\omega^{n+1}(t)\|_{L^{2}}^{2}\\&\leq C(\varepsilon,T)(\|f^{n}(t)-f^{n-1}(t)\|_{L^{\infty}}^{2}+\|\omega^{n+1}(t)\|_{L^{2}}^{2}+\|\omega^{n}(t)\|_{L^{2}}^{2}).\end{aligned}$$
		
		Please refer to subsection 3.1.3 in \cite{ref29} for a detailed derivation process. Then, we now use \eqref{eq4.36} to find
		$$\begin{aligned}\frac{1}{2}\frac{\mathrm{d}}{\mathrm{d}t}\|\omega^{n+1}(t)\|_{L^{2}}^{2}+\frac{1}{2}\|\nabla\omega^{n+1}(t)\|_{L^{2}}^{2}\leq C(\varepsilon,T)\biggl(\int_{0}^{t}\|\omega^{n}(s)\|_{L^{2}}^{2}\:\mathrm{d}s+\|\omega^{n+1}(t)\|_{L^{2}}^{2}+\|\omega^{n}(t)\|_{L^{2}}^{2}\biggr).\end{aligned}$$
		
		Therefore, by integrating the previous inequality with respect to $t$ and using 
		$$\int_0^t\int_0^s\|\omega^n(r)\|_{L^2(\mathbb{T}^3)}^2dr\mathrm{d}s\leq T\int_0^t\|\omega^n(r)\|_{L^2(\mathbb{T}^3)}^2dr,$$
		we can reach the desired result.
	\end{proof}
	\begin{lem}\label{lem4.7}
		The unique solution sequence $[f^{n}, u^{n}]$ obtained in Lemma \ref{lem4.5} strongly converges in $L^{\infty }(\mathbb{T} ^3\times \mathbb{R} ^3 \times (0, T) ) \times L^{\infty }( 0, T; L^2( \mathbb{T} ^3) )$ towards the weak solution of \eqref{eq4.5}.
	\end{lem}
	\begin{proof}
		The proof simply follows from Lemma \ref{lem4.6}. Indeed, it follows from Lemma \ref{lem4.6}(2) and the Grönwall-type inequality in Proposition \ref{prop2.1}  that
		$$\|u^n-u^{n-1}\|_{L^\infty(0,T;L^2(\mathbb{T}^3))}^2\leq\frac{K^nT^n}{n!},$$
		for some $K>0.$ This gives the strong convergence of $\{u^n\}$ in $L^\infty(0,T;L^2(\mathbb{T}^3))$, and subsequently, this also yields
		$$\|f^n(t)-f^{n-1}(t)\|_{L^\infty}+\|z^n(t)-z^{n-1}(t)\|_{L^\infty}\leq\frac{(\tilde{K})^nT^n}{n!},$$
		for some $\tilde{K}>0.$ The convergence of $[f^n,u^n]$ in $L^\infty((0,T)\times\mathbb{T}^3\times\mathbb{R}^3)\times L^\infty(0,T;L^2(\mathbb{T}^3))$ enables us to pass to the limit with respect to $n$ in the weak formulation for the approximated system. This completes the proof.
	\end{proof}
	\subsection{Proof of Theorem 1.2} 
	So far, we have already proved that the sequence $[f^n_\varepsilon,u^n_\varepsilon]$ is converged and we note its limit as $[f_\varepsilon,u_\varepsilon]$. From Lemma \ref{lem4.7}, we know that is a solution of the regularized system \eqref{eq4.5}. It is easy to observe that $[f_\varepsilon,u_\varepsilon]$ will formally converge to the solution of \eqref{eq1} as $\varepsilon \rightarrow 0$. We will strictly prove this in this subsection.
	\begin{lem}\label{lem4.8}
		Let $\left[ f_{\varepsilon }, u_{\varepsilon }\right]$ be a solution to \eqref{eq4.5}. Then there exists $T^*\in(0,T]$ and a positive constant $C$ which is independent of $\varepsilon$ such that
		$$\begin{aligned}&(1)\quad\|f_{\varepsilon}\|_{L^{\infty}(\mathbb{T}^{3}\times\mathbb{R}^{3}\times(0,T))}\leq C(T)\|f^{\mathrm{in}}\|_{L^{\infty}(\mathbb{T}^{3}\times\mathbb{R}^{3})},\\&(2)\quad\|u_{\varepsilon}(t)\|_{L^{2}(\mathbb{T}^{3})}^{2}+\int_{0}^{t}\|\nabla u_{\varepsilon}(s)\|_{L^{2}(\mathbb{T}^{3})}^{2}\mathrm{d}s+M_{2}f_{\varepsilon}(t)\leq C,\quad\mathrm{a.e.}\:t\leq T^{*},\end{aligned}$$
	\end{lem}
	\begin{proof}
		Since $\|f_{\varepsilon}^{\mathrm{in}}\|_{L^{\infty}}\leq\|f^{\mathrm{in}}\|_{L^{\infty}}$, we have
		$$\|f_\varepsilon\|_{L^\infty((0,T)\times\mathbb{T}^3\times\mathbb{R}^3)}\leq{C(T)}\|f^{\mathrm{in}}\|_{L^\infty(\mathbb{T}^3\times\mathbb{R}^3)},$$
		where we have used Lemma \ref{lem4.5}(1).
		
		(2) We use Proposition \ref{prop2.2} to estimate $M_{2}f_{\varepsilon}$ as follows.
		$$\begin{aligned}\frac{\mathrm{d}}{\mathrm{d}t}M_{2}f_{\varepsilon}&=2\int_{\mathbb{T}^3\times\mathbb{R}^3}w\cdot F[f_\varepsilon,u_\varepsilon]f_\varepsilon\:\mathrm{d}x\mathrm{d}w\\
			&=2\int_{\mathbb{T}^3\times\mathbb{R}^3}w\cdot(m_\varepsilon*u_\varepsilon)f_\varepsilon\:\mathrm{d}x\mathrm{d}w+2\int_{\mathbb{T}^3\times\mathbb{R}^3}w\cdot F[f_\varepsilon]f_\varepsilon\:\mathrm{d}x\mathrm{d}w-2M_2f_\varepsilon\\
			&\leq2\|m_\varepsilon*u_\varepsilon\|_{L^5(\mathbb{T}^3)}\|m_1f_\varepsilon\|_{L^{5/4}(\mathbb{T}^3)}-2M_2f_\varepsilon\\
			&\leq C_1\|u_\varepsilon\|_{L^5(\mathbb{T}^3)}(M_2f_\varepsilon)^{4/5}-2M_2f_\varepsilon,
		\end{aligned}$$
		where we used
		$$m_1f_\varepsilon\leq C(m_2f_\varepsilon)^\frac{4}{5} \Rightarrow \|m_1f_\varepsilon\|_{L^{5/4}(\mathbb{T}^3)}\leq C(M_2f_\varepsilon)^{4/5},$$
		and
		$$\begin{aligned}&\int_{\mathbb{T}^3\times\mathbb{R}^3}w\cdot F[f_\varepsilon]f_\varepsilon\:\mathrm{d}x\mathrm{d}w=\int_{\mathbb{T}^6\times\mathbb{R}^6}\phi(|x-x_*|)w\cdot (\hat{v}(w_*)-\hat{v}(w)f_\varepsilon(t,x,w)f_\varepsilon(t,x_*,w_*)\:\mathrm{d}x\mathrm{d}x_*\mathrm{d}w\mathrm{d}w_*\\
			&=-\frac{1}{2}\int_{\mathbb{T}^6\times\mathbb{R}^6}\phi(|x-x_*|)(w_*-w)\cdot (\hat{v}(w_*)-\hat{v}(w)f_\varepsilon(t,x,w)f_\varepsilon(t,x_*,w_*)\:\mathrm{d}x\mathrm{d}x_*\mathrm{d}w\mathrm{d}w_*\\
			&\leq-\frac{C(c)}{2}\int_{\mathbb{T}^6\times\mathbb{R}^6}\phi(|x-x_*|)|\hat{v}(w_*)-\hat{v}(w)|^2f_\varepsilon(t,x,w)f_\varepsilon(t,x_*,w_*)\:\mathrm{d}x\mathrm{d}x_*\mathrm{d}w\mathrm{d}w_*\leq 0. \end{aligned}$$
		Note that $C_1(\|f^{\mathrm{in}}\|_{L^\infty})>0$ is independent of $\varepsilon.$ Set
		$h(t):=\left(M_2f_\varepsilon(t)\right)^{1/5}.$ Then $h$ satisfies
		$$5h'(t)\leq C_1\|u_\varepsilon\|_{L^5}-2h.$$
		We apply Grönwall's Lemma to yield
		\begin{equation}\label{eq4.37}
			M_2f_\varepsilon(t)\leq C\left(1+\int_0^t\|u_\varepsilon(s)\|_{L^5}\:\mathrm{d}s\right)^5,
		\end{equation}
		where $C>0$ is independent of $\varepsilon$. For the estimate of fluid part, we use a similar argument as in Section 3.3 of \cite{ref29} together with the estimate of $f_\varepsilon$ to obtain that there exist a $T^*>0$ and a positive constant $C$ which is independent of $\varepsilon$ such that
		\begin{equation}\label{eq4.38}
			\|u_\varepsilon(t)\|_{L^2(\mathbb{T}^3)}^2+\int_0^t\|\nabla u_\varepsilon(s)\|_{L^2(\mathbb{T}^3)}^2 \mathrm{d}s\leq C,\quad t\leq T^*.
		\end{equation}
		We combine \eqref{eq4.37} and \eqref{eq4.38} to achieve the desired result. 
	\end{proof}
	
	It follows from the uniform bound estimates in Lemma \ref{lem4.8} that there exists $( f, u) \in L^\infty ( \mathbb{T} ^3\times \mathbb{R} ^3\times \mathbb{R} _+ \times ( 0, T^* ) ) \times ( L^\infty ( 0, T^* ; L^2( \mathbb{T} ^3) ) \cap L^2( 0, T^* ; H^1( \mathbb{T} ^3) ) \cap \mathcal{C} ^0( [ 0, T^* ] ; \mathcal{V} ^{\prime })$ such that up to a subsequence, $(f_\varepsilon)$ weakly star converges to $f$ in $L^\infty(\mathbb{T}^3\times\mathbb{R}^3\times(0,T^*))$ and $(u_\varepsilon)$ weakly converges to $u$ in $L^2(0,T^*;H^1(\mathbb{T}^3))$ and weakly star converges in $L^{\infty}(0,T^*;L^2(\mathbb{T}^3)).$ Also, $(u_\varepsilon)$ converges strongly to $u$ in $L^2(0,T^*;L^2(\mathbb{T}^3))$, since $(\partial_tu_\varepsilon)$ is uniformly bounded in $L^{3/2}(0,T^*;\mathcal{V}^{\prime}).$ These convergences allow us to pass to the limit in the weak formulation of the regularized system \eqref{eq4.5}, and this concludes the local-in-time existence of weak solutions to the system \eqref{eq4.5}.
	
	We finally need to extend these local solutions to the global ones. For this, we need the following total classical energy estimates.
	\begin{lem}\label{lem4.9}
		Let $(f,u)$ be a weak solution to \eqref{eq1}. Then, $(f,u)$ satisfies the following estimates
		$$\begin{aligned}
			& (1)\quad\frac{1}{2}M_{2}f(t)+\frac{1}{2}\|u(t)\|_{L^{2}(\mathbb{T}^{3})}^{2}+\int_{0}^{t}\|\nabla u(s)\|_{L^{2}(\mathbb{T}^{3})}^{2}\mathrm{d}s+\int_{0}^{t}\int_{\mathbb{T}^{3}\times\mathbb{R}^{3}}f|u-w|^{2}\mathrm{d}x\mathrm{d}w\mathrm{d}s \\
			& \leq \frac{1}{2}\|u^{\mathrm{in}}\|_{L^{2}(\mathbb{T}^{3})}^{2}+\frac{1}{2}M_{2}f^{\mathrm{in}}, \\
			& (2)\quad \operatorname*{sup}_{0\leq t\leq T^{*}}\|f(t)\|_{L^{\infty}(\mathbb{T}^{3}\times\mathbb{R}^{3})}\leq\mathrm{e}^{CT}\|f^{\mathrm{in}}\|_{L^{\infty}(\mathbb{T}^{3}\times\mathbb{R}^{3})}, \;\mathrm{a.e.}t\;\leq T^{*}.
		\end{aligned}$$
	\end{lem}
	\begin{proof}
		First, we multiply \eqref{eq4.5}$_1$ by $\frac12|w|^{2}$ and integrate the resulting relation over $\mathbb{T}^3\times\mathbb{R}^3\times(0,t)$, and we also multiply (4.1)$_2$ by $u_\varepsilon$ and integrate over $\mathbb{T}^3\times(0,t).$ We combine those two equations to get
		\begin{align}\label{eq4.39}
			&\frac{1}{2}M_{2}f_{\varepsilon}(t)+\frac{1}{2}\|u_{\varepsilon}(t)\|_{L^{2}(\mathbb{T}^{3})}^{2}+\int_{0}^{t}\|\nabla u_{\varepsilon}(s)\|_{L^{2}(\mathbb{T}^{3})}^{2}\:\mathrm{d}s+\int_{0}^{t}\int_{\mathbb{T}^{3}\times\mathbb{R}^{3}}f_{\varepsilon}|u_{\varepsilon}-w|^{2}\:\mathrm{d}x\mathrm{d}w\mathrm{d}s\nonumber\\
			&=\frac{1}{2}M_{2}f_{\varepsilon}^{\mathrm{in}}+\frac{1}{2}\|u_{\varepsilon}^{\mathrm{in}}\|_{L^{2}(\mathbb{T}^{3})}^{2}+\int_{0}^{t}\int_{\mathbb{T}^{3}\times\mathbb{R}^{3}}w\cdot F[f_{\varepsilon}]f_{\varepsilon}\:\mathrm{d}x\mathrm{d}w\mathrm{d}s+R_{\varepsilon}(t),
		\end{align}
		where
		$$\begin{aligned}R_{\varepsilon}(t)&:=\int_{0}^{t}\int_{\mathbb{T}^{3}\times\mathbb{R}^{3}}f_{\varepsilon}|u_{\varepsilon}|^{2}(1-\gamma_{\varepsilon}(w))\:\mathrm{d}x\mathrm{d}w\mathrm{d}s-\int_{0}^{t}\int_{\mathbb{T}^{3}\times\mathbb{R}^{3}}f_{\varepsilon}u_{\varepsilon}\cdot w(1-\gamma_{\varepsilon}(w))\:\mathrm{d}x\mathrm{d}w\mathrm{d}s\\&+\int_{0}^{t}\int_{\mathbb{T}^{3}\times\mathbb{R}^{3}}f_{\varepsilon}(u_{\varepsilon}-m_{\varepsilon}*u_{\varepsilon})\cdot w\:\mathrm{d}x\mathrm{d}w\mathrm{d}s=:R^1_\varepsilon(t)+R^2_\varepsilon(t)+R^3_\varepsilon(t).\end{aligned}$$
		Note that
		$$\int_0^t\int_{\mathbb{T}^3\times\mathbb{R}^3}w\cdot F[f_\varepsilon]f_\varepsilon\:\mathrm{d}x\mathrm{d}w\mathrm{d}s\leq\ 0.$$
		
		On the other hand, we can use the almost same argument as in the proof of the Lemma 2 of \cite{ref29} together with the uniform support estimate of $f_\varepsilon$ to get
		$$\sup\limits_{0\leq t\leq T^*}R_\varepsilon(t)\to0\quad\text{as}\quad\varepsilon\to0.$$
		
		We apply the above observations to \eqref{eq4.39} to obtain
		$$\begin{aligned}\frac{1}{2}M_{2}f_{\varepsilon}(t)&+\frac{1}{2}\|u_{\varepsilon}(t)\|_{L^{2}(\mathbb{T}^{3})}^{2}+\int_{0}^{t}\|\nabla u_{\varepsilon}(s)\|_{L^{2}(\mathbb{T}^{3})}^{2}\:\mathrm{d}s+\int_{0}^{t}\int_{\mathbb{T}^{3}\times\mathbb{R}^{3}}f_{\varepsilon}|u_{\varepsilon}-w|^{2}\:\mathrm{d}x\mathrm{d}w\mathrm{d}s\\&\leq\frac{1}{2}M_{2}f_{\varepsilon}^{\mathrm{in}}+\frac{1}{2}\|u_{\varepsilon}^{\mathrm{in}}\|_{L^{2}(\mathbb{T}^{3})}^{2}+o(\varepsilon).\end{aligned}$$
		Let $\varepsilon\to 0$ in the above equation and use the strategy of the Lemma 2 in \cite{ref29} again, we can obtain the desired result.
	\end{proof}
	It remains to extend the local solutions obtained in the above to the global ones. However, 
	this can be done by employing the same method used in Section 3.6 of \cite{ref29} together with Lemma \ref{lem4.9}.
	\section{Conclusion}
	In this article, we propose a coupled model of the relativistic Cucker-Smale equation and the incompressible Navier-Stokes equation based on the theory of friction. We also analyze its asymptotic stability and the existence of weak solutions. According to our knowledge, there have been no previous studies on this coupled model.
	
	Firstly, we have proved that the smooth solution of the RCS-NS equation exhibits exponential stability under certain conditions. Physically, this means that the particle velocity will converge with the fluid's velocity. In addition, we also proposed a weak solution form for the RCS-NS equation and proved its existence under conditions of large initial values. This provides a potential theoretical basis for designing its numerical algorithms.
	
	In the future, we will continue to investigate the existence and uniqueness of strong solutions in the periodic region of the RCS-NS model, as well as analyze the behavior of solutions in the whole domain $\mathbb {R}^n $. In addition, further exploration is needed on the coupling structure between RCS equation and compressible Navier-Stokes equations, relativistic Euler equations, Stokes equation and other fluid equations.
	
	{\bf Acknowledgments.}  Weiyuan Zou is supported by the National Natural Science Foundation of China (NSFC)12001033.
	\vskip 0.3cm {\bf Conflict of Interest}\quad The authors declare no
	conflict of interest.


\begin{thebibliography}{99} 
		\bibitem{ref1} J. Toner, Y. Tu. Flocks, herds, and schools: A quantitative theory of flocking[J]. Physical Review E, 1998, 58(4): 4828-4858.
		
		\bibitem{ref2} C. M. Topaz, A. L. Bertozzi. Swarming patterns in a two-dimensional kinematic model for biological groups[J]. SIAM Journal on Applied Mathematics, 2004, 65: 152–174. 
		
		\bibitem{ref3} A. T. Winfree. Biological rhythms and the behavior of populations of coupled oscillators[J]. Journal of Theoretical Biology, 1967, 16(1):15-42. 
		
		\bibitem{ref4} J. Buck, E. Buck. Biology of Synchronous Flashing of Fireflies[J]. Nature, 1966, 211(5049): 562-564.
		
		\bibitem{ref5} S. Y. Ha, J. Jeong, S.E. Noh, et al. Emergent dynamics of Cucker-Smale flocking particles in a random environment[J]. Journal of Differential Equations, 2017, 262(3): 2554-2591. 
		
		\bibitem{ref6} F. Cucker, S. Smale. Emergent Behavior in Flocks[J]. IEEE Transactions on Automatic Control, 2007, 52(5): 852-862.
		
		\bibitem{ref7} D. Bhaya. Light matters: phototaxis and signal transduction in unicellular cyanobacteria[J]. Molecular Microbiology, 2010, 53(3): 745-754. 
		
		\bibitem{ref8} A. Jakob, N. Schuergers, A. Wilde. Phototaxis assays of synechocystis sp PCC 6803 at macroscopic and microscopic scales[J]. Bio-protocol, 2017, 7(11): e2328.
		
		\bibitem{ref9} S. Y. Ha, T. Ruggeri, Emergent dynamics of a thermodynamically consistent particle model[J]. Archive for Rational Mechanics and Analysis, 2017, 223: 1397-1425.
		
		\bibitem{ref10} S. Y. Ha, J. Kim, T. Ruggeri, From the relativistic mixture of gases to the relativistic Cucker-Smale flocking[J]. Archive for Rational Mechanics and Analysis, 2020, 235: 1661-1706.
		
		\bibitem{ref11} H. Ahn, S.-Y. Ha, J. Kim, Uniform stability of the relativistic Cucker-Smale model and its application to a mean-field limit[J]. Communications on Pure and Applied Analysis, 2021, 20: 4209-4237.
		
		\bibitem{ref12} J. Byeon, S. Y. Ha, J. Kim. Asymptotic flocking dynamics of a relativistic Cucker-Smale flock under singular communications[J]. Journal of Mathematical Physics, 2022, 63: 1-17.
		
		\bibitem{ref13} H. Ahn. Asymptotic flocking of the relativistic Cucker–Smale model with time delay[J]. Networks and Heterogeneous Media, 2023, 18(1): 29-47.
		
		\bibitem{ref14} C. Baranger, L. Boudin, P.E Jabin, S. Mancini, A modeling of biospray for the upper airways[J]. Mathematics and applications to biology and medicine, 2005, 14: 41-47.
		
		\bibitem{ref15} L. Boudin, L. Desvillettes, R. Motte. A Modeling of Compressible Droplets in a Fluid[J]. Communications in mathematical sciences, 2003, 1(4): 657-669.
		
		\bibitem{ref16} R. Cafisch, G. C. Papanicolaou. Dynamic theory of suspensions with Brownian effects[J]. Siam Journal on Applied Mathematics, 1983, 43: 885-906.
		
		\bibitem{ref17} L. Desvillettes, F. Golse, V. Ricci. The mean-field limit for solid particles in a Navier-Stokes fow[J]. Journal of Statistical Physics, 2008, 131: 941-967.
		
		\bibitem{ref18} P. E. Jabin, B. Perthame. Notes on mathematical problems on the dynamics of dispersed particles interacting through a fluid[J]. Boston, Birkhauser, 2000: 111-147.
		
		\bibitem{ref19} P. J. O'Rourke. Collective Drop Effects on Vaporizing Liquid Sprays[D]. Princeton University, 1981.
		
		\bibitem{ref20} F. A. Williams. Spray combustion and atomization[J]. Phys fluids, 1958, 1: 541-545.
		
		\bibitem{ref21} Combustion theory[M]. Addison-Wesley,1965.
		
		\bibitem{ref22} A. A. Amsden, P. J. Orourke, T. D. Butler. KIVA-2: A computer program for chemically reactive flows with sprays[J]. Nasa Sti/recon Technical Report N, 1989.
		
		\bibitem{ref23} C. Baranger, L. Boudin, P. E. Jabin, et al. A modeling of biospray for the upper airways[J]. Esaim Proceedings, 2005, 14: 41-47.
		
		\bibitem{ref24} S. Berres, R. Biuirger. K. H. Karlsen, et al. Strongly degenerate parabolic-hyperbolic systems modeling polydisperse sedimentation with compression[J]. Siam Journal On Applied Mathematics, 2003, 64: 41-80.
		
		\bibitem{ref25} S. Berres, R. Biuirger, E. M. Tory. Mathematical model and numerical simulation of the liquid fluidization of polydisperse solid particle mixtures[J]. Computing and Visualization in Science, 2004, 6(2/3): 67-74.
		
		\bibitem{ref26} N. Bellomo, S. Y, Ha, A quest toward a mathematical theory of the dynamics of swarms[J]. Mathematical models and methods in applied sciences, 2017, 27: 745-770.
		
		\bibitem{ref27} F. A. Williams. Combustion Theory, The Fundamental Theory of Chemically Reacting Flow Systems. Addison-Wesley series in engineering science, 1965.
		
		\bibitem{ref28} H. Bae, Y. P. Choi, S. Y. Ha, et al. Time-asymptotic interaction of flocking particles and an incompressible viscous fluid[J]. Nonlinearity, 2012, 24: 1155-1177.
		
		\bibitem{ref29} L. Boudin, L. Desvillettes, C. Grandmont, et al. Global existence of solution for the coupled Vlasov and Navier-Stokes equations[J]. Differential and Integral Equations, 2009, 22(11-12): 1247-1271.
		
		\bibitem{ref30} H, Bae, Y.-P. Choi, S.-Y. Ha, et al. Global existence of strong solution for the Cucker-Smale-Navier-Stokes system[J]. Journal of Differential Equations, 2014, 257: 2225-2255.
		
		\bibitem{ref31} Y. P. Choi, S. Y. Ha, J. Jung, et al. Global dynamics of the thermomechanical Cucker-Smale ensemble immersed in incompressible viscous fluids[J]. Nonlinearity, 2019, 32: 1597-1640.
		
		\bibitem{ref32} H. Ahn, S. Y. Ha, J. Kim. Nonrelativistic limits of the relativistic Cucker–Smale model and its kinetic counterpart[J]. Journal of mathematical physics, 2022, 63: 082701.
		
		\bibitem{ref33} D. O. Hebb. The organization of behavior[J]. MIT Press, 1988.
		
	\end{thebibliography}
\end{document}